\documentclass[12pt]{amsart}

\usepackage{a4wide}
\usepackage{amsmath,amsthm,amssymb}
\usepackage{hyperref}
\usepackage{graphicx,xcolor}
\usepackage{tikz,ytableau}
\usetikzlibrary{decorations.pathreplacing}
\ytableausetup{aligntableaux=bottom}
\numberwithin{equation}{section}

\newtheorem{thm}{Theorem}[section]
\newtheorem{lem}[thm]{Lemma}
\newtheorem{prop}[thm]{Proposition}
\newtheorem{cor}[thm]{Corollary}
\theoremstyle{definition}

\newtheorem{defn}[thm]{Definition}

\newtheorem{remark}[thm]{Remark}

\newcommand\wg{\widetilde{g}}

\newcommand\wb{\widetilde{\beta}}
\newcommand\ove{\overline{e}}
\newcommand\ovh{\overline{h}}
\newcommand\ovH{\overline{H}}

\newcommand\col{\operatorname{col}}
\newcommand\row{\operatorname{row}}

\newcommand\ZZ{\mathbb{Z}}
\newcommand\NN{\mathbb{N}}
\newcommand\Par{\operatorname{Par}}
\newcommand\RPar{\operatorname{RPar}}
\newcommand\QPar{\operatorname{QPar}}

\newcommand\RPP{\operatorname{RPP}}

\newcommand\lm{{\lambda/\mu}}

\newcommand\wt{\operatorname{wt}}
\newcommand\R{\mathcal{R}}
\newcommand\oR{\overline{\R}}
\newcommand\EE{\mathcal{E}}
\newcommand\HH{\mathcal{H}}
\newcommand\owt{\overline{\wt}}

\title[Jacobi--Trudi formulas for flagged dual Grothendieck
polynomials]{Jacobi--Trudi formulas for flagged refined dual stable Grothendieck
  polynomials}

\author{Jang Soo Kim}
\thanks{} 
\address{Department of Mathematics,
Sungkyunkwan University (SKKU), Suwon, Gyeonggi-do 16419, South Korea}
\email{jangsookim@skku.edu}

\begin{document}

\begin{abstract}
  Recently Galashin, Grinberg, and Liu introduced the refined dual stable
  Grothendieck polynomials, which are symmetric functions in $x=(x_1,x_2,\dots)$
  with additional parameters $t=(t_1,t_2,\dots)$. The refined dual stable
  Grothendieck polynomials are defined as a generating function for reverse
  plane partitions of a given shape. They interpolate between Schur functions
  and dual stable Grothendieck polynomials introduced by Lam and Pylyavskyy in
  2007. Flagged refined dual stable Grothendieck polynomials are a more refined
  version of refined dual stable Grothendieck polynomials, where lower and upper
  bounds are given for the entries of each row or column. In this paper
  Jacobi--Trudi-type formulas for flagged refined dual stable Grothendieck
  polynomials are proved using plethystic substitution. This resolves a
  conjecture of Grinberg and generalizes a result by Iwao and
  Amanov--Yeliussizov.
\end{abstract}

\maketitle
% \tableofcontents

\section{Introduction}

The (skew) Schur functions $s_{\lm}(x)$ are a central object in algebraic
combinatorics. They are symmetric functions in the variables $x=(x_1,x_2,\dots)$
and can be defined combinatorially as a generating function for semistandard
Young tableaux of shape $\lm$. The Jacobi--Trudi formula and its dual formula
express $s_{\lm}(x)$ as a determinant in terms of the complete homogeneous
symmetric functions $h_k(x)$ and the elementary symmetric functions $e_k(x)$,
respectively:
\begin{align}
  \label{eq:s=det h}
  s_{\lambda/\mu}(x) &= \det
  \left( h_{\lambda_i-\mu_j-i+j} (x) \right)_{1\le i,j\le \ell(\lambda)},\\
  \label{eq:s=det e}
  s_{\lambda'/\mu'}(x) &= \det
  \left( e_{\lambda_i-\mu_j-i+j} (x) \right)_{1\le i,j\le \ell(\lambda)},
\end{align}
where $\ell(\lambda)$ is the number of parts in $\lambda$ and $\lambda'$ is the
transpose of $\lambda$.

The row-flagged and column-flagged Schur functions
$s_{\lambda/\mu}^{\row(\alpha,\beta)}(x)$ and
$s_{\lambda/\mu}^{\col(\alpha,\beta)}(x)$ are defined as a generating function
for semistandard Young tableaux of shape $\lm$ in which entries in each row or
column have lower and upper bounds specified by $\alpha$ and $\beta$. Flagged
Schur functions were introduced by Lascoux and Sch\"utzenberger \cite{LS1982p}
in their study of Schubert polynomials. See \cite{Chen2002, Merzon_2015,
  Wachs_1985} and references therein for more details on flagged Schur
functions. Jacobi--Trudi formulas for flagged Schur functions were discovered by
Gessel \cite{Gessel_unpublished} and Wachs \cite{Wachs_1985}.

\begin{thm}\cite[Theorems~3.5 and 3.5*]{Wachs_1985}
\label{thm:wachs}
Let $\alpha=(\alpha_1,\dots,\alpha_n)$ and $\beta=(\beta_1,\dots,\beta_n)$ be
sequences of nonnegative integers and let $\lambda$ and $\mu$ be partitions with
at most $n$ parts.

If $\alpha_i\le \alpha_{i+1}$ and $\beta_i\le \beta_{i+1}$ whenever
$\mu_i<\lambda_{i+1}$, then
\begin{equation}
  \label{eq:wachs1}
s_{\lambda/\mu}^{\row(\alpha,\beta)}(x)  = \det \left(
    h_{\lambda_i-\mu_j-i+j}(x_{\alpha_j+1},\dots, x_{\beta_i})
  \right)_{1\le i,j\le n},
\end{equation}
where $(x_{\alpha_j+1},\dots, x_{\beta_i})$ is the empty list if $\beta_i \leq
\alpha_j$.

If $\alpha_i-\mu_i\le \alpha_{i+1}-\mu_{i+1}+1$ and $\beta_i-\lambda_i\le
\beta_{i+1}-\lambda_{i+1}+1$ whenever $\mu_i<\lambda_{i+1}$, then
\begin{equation}
 \label{eq:wachs2}
s_{\lambda'/\mu'}^{\col(\alpha,\beta)}(x)  = \det \left(
    e_{\lambda_i-\mu_j-i+j}(x_{\alpha_j+1},\dots, x_{\beta_i})
  \right)_{1\le i,j\le n}.
\end{equation}
\end{thm}

Lam and Pylyavskyy \cite{LP2007} introduced dual stable Grothendieck polynomials
$g_{\lm}(x)$, which originate from the $K$-theory of Grassmannians. They showed
that $g_{\lm}(x)$ is a generating function for reverse plane partitions of shape
$\lm$. The refined dual stable Grothendieck polynomials $\wg_{\lm}(x;t)$ are
power series in variables $x=(x_1,x_2,\dots)$ and $t=(t_1,t_2,\dots)$ introduced
by Galashin, Grinberg, and Liu \cite{GGL2016}. Similarly to dual stable
Grothendieck polynomials, $\wg_{\lm}(x;t)$ are defined as a generating function
for reverse plane partitions of shape $\lm$ with more refined weight system. The
refined dual stable Grothendieck polynomials interpolate between Schur functions
and dual stable Grothendieck polynomials. If $t_i=0$ for all $i$, then
$\wg_{\lm}(x;t)$ becomes the Schur function $s_{\lm}(x)$, and if $t_i=1$ for all
$i$, then $\wg_{\lm}(x;t)$ becomes the dual stable Grothendieck polynomial
$g_{\lm}(x)$.

The following theorem was conjectured by Grinberg \cite{Grinberg_conj} and
proved independently by Amanov and Yeliussizov \cite{AmanovYeliussizov}, and the
author \cite{Kim:JT}.

\begin{thm}\label{thm:JT}\cite{AmanovYeliussizov, Kim:JT}
For partitions $\lambda$ and $\mu$, we have
\[
  \widetilde{g}_{\lambda/\mu}(x;t) = \det
  \left( e_{\lambda'_i-\mu'_j-i+j}
    (x_1,x_2,\dots,t_{\mu'_j+1},t_{\mu'_j+2},\dots,t_{\lambda'_i-1})
  \right)_{1\le i,j\le \ell(\lambda')},
\]
where, if $\mu'_j+1>\lambda'_i-1$, the $(i,j)$ entry is defined to be
$e_{\lambda'_i-\mu'_j-i+j} (x_1,x_2,\dots)$.
\end{thm}

Since there are two Jacobi--Trudi formulas for $s_{\lm}(x)$ in \eqref{eq:s=det
  h} and \eqref{eq:s=det e}, a natural question is whether there is a
Jacobi--Trudi formula for $\wg_\lm(x;t)$ in terms of $h_k$'s. For the case of
dual stable Grothendieck polynomials, equivalently the case that all $t_i=1$,
Amanov and Yeliussizov \cite[Theorem~14]{AmanovYeliussizov}, and Iwao
\cite[Proposition~5.2]{iwao20:free_groth} found the following formula.

\begin{thm}\label{thm:JT2}
\cite{AmanovYeliussizov, iwao20:free_groth}
For partitions $\lambda$ and $\mu$, we have
\[
  g_\lm(x)=\widetilde{g}_{\lambda/\mu}(x;(1,1,\dots)) = \det
  \left( \phi^{i-j} h_{\lambda_i-\mu_j-i+j} (x)
  \right)_{1\le i,j\le \ell(\lambda)},
\]
where $\phi^k h_n = \sum_{i=0}^n \binom{k+i-1}{i} h_{n-i}$.
\end{thm}

In this paper we give a Jacobi--Trudi formula for $\wg_\lm(x;t)$ in terms of
$h_k$'s using plethystic substitution. We also give an equivalent version of
Theorem~\ref{thm:JT} using plethystic substitution. More generally, we prove
Jacobi--Trudi formulas for \emph{flagged} refined dual stable Grothendieck
polynomials $\wg_{\lambda'/\mu'}^{\col(\alpha,\beta)}(x;t)$ and
$\wg_{\lambda/\mu}^{\row(\alpha,\beta)}(x;t)$, which are generating functions
for reverse plane partitions in which each column and row has lower and upper
bounds given by $\alpha$ and $\beta$. See Section~\ref{sec:preliminaries} for
the precise definitions.

Our main results are the two Jacobi--Trudi-type formulas in the following
theorem.

\begin{thm} \label{thm:main} Let $\alpha=(\alpha_1,\dots,\alpha_n)$ and
  $\beta=(\beta_1,\dots,\beta_n)$ be sequences of nonnegative
  integers and let $\lambda$ and $\mu$ be partitions with at most $n$ parts. 

  If $\alpha_i\le \alpha_{i+1}+1$ and $\beta_i\le \beta_{i+1}+1$ whenever
  $\mu_i<\lambda_{i+1}$, then
  \begin{equation}
\label{eq:main1}
  \wg_{\lambda'/\mu'}^{\col(\alpha,\beta)}(x;t) = \det \left(
    e_{\lambda_i-\mu_j-i+j}[X_{(\alpha_j,\beta_i]}+T_{\lambda_i-1}-T_{\mu_j}]
  \right)_{1\le i,j\le n},
  \end{equation}
where $X_{(i,j]} = x_{i+1}+x_{i+2}+\dots+x_j$ for $i<j$ and $X_{(i,j]} = 0$ for
$i\ge j$, and $T_i=t_1+t_2+\dots+t_i$ for $i\ge1$ and $T_0=0$.

If $\alpha_i\le \alpha_{i+1}$ and $\beta_i\le \beta_{i+1}$ whenever
$\mu_i<\lambda_{i+1}$, then
\begin{equation}
\label{eq:main2}
\wg_{\lambda/\mu}^{\row(\alpha,\beta)}(x;t) = \det \left(
      h_{\lambda_i-\mu_j-i+j}[X_{(\alpha_j,\beta_i]}+T_{i-1}-T_{j-1}]
      \right)_{1\le i,j\le n}.
\end{equation}
\end{thm}

Note that the assumption on $\alpha$ and $\beta$ in our formula \eqref{eq:main1}
is different from that in the formula \eqref{eq:wachs2}. In fact
\eqref{eq:main1} is not true under the assumptions for \eqref{eq:wachs2}, see
Remark~\ref{rem:wachs}.

The basic idea of proof of \eqref{eq:main1} and \eqref{eq:main2} is to show that
both sides of the equation satisfy the same recurrence relation. We also show
that \eqref{eq:main1} is equivalent to the following formula, which was
conjectured by Grinberg (private communication).

\begin{thm}\label{thm:col_flag3}
  Let $\alpha=(\alpha_1,\dots,\alpha_n)$ and $\beta=(\beta_1,\dots,\beta_n)$ be
  sequences of nonnegative integers and let $\lambda$ and $\mu$ be
  partitions with at most $n$ parts. If $\alpha_i\le \alpha_{i+1}+1$ and
  $\beta_i\le \beta_{i+1}+1$ whenever $\mu_i<\lambda_{i+1}$, then
\begin{equation}
\wg_{\lambda'/\mu'}^{\col(\alpha,\beta)}(x;t)  = \det \left(
    e_{\lambda_i-\mu_j-i+j}(x_{\alpha_j+1},\dots, x_{\beta_i},t_{\mu_j+1},\dots,t_{\lambda_i-1})
  \right)_{1\le i,j\le n}.
\end{equation}
\end{thm}

Note that Theorem~\ref{thm:JT} follows from Theorem~\ref{thm:col_flag3}. As a
corollary of Theorem~\ref{thm:main} we obtain two Jacobi--Trudi formulas for the
refined dual stable Grothendieck polynomials.

\begin{cor}\label{cor:main}
Let   $\lambda$ and $\mu$ be partitions with at most $n$ parts. Then
\begin{align}
  \label{eq:jt1}
\wg_{\lambda'/\mu'}(x;t)&= \det \left(
      e_{\lambda_i-\mu_j-i+j}[X+T_{\lambda_i-1}-T_{\mu_j}]
      \right)_{1\le i,j\le n},\\
  \label{eq:jt2}
\wg_{\lambda/\mu}(x;t)&= \det \left(
      h_{\lambda_i-\mu_j-i+j}[X+T_{i-1}-T_{j-1}]
      \right)_{1\le i,j\le n},
\end{align}
where $X=x_1+x_2+\cdots$ and $T_i=t_1+t_2+\dots+t_i$ for $i\ge1$ and $T_0=0$.
\end{cor}

It can also be shown that the formula \eqref{eq:jt1} is equivalent to
Theorem~\ref{thm:JT}. Note that Corollary~\ref{cor:main} reproves the symmetry
of $\wg_{\lambda/\mu}(x;t)$ in the $x$ variables.

We note that Motegi and Scrimshaw \cite{Motegi} also proved \eqref{eq:jt2} using
difference operators.

The remainder of this paper is organized as follows. In
Section~\ref{sec:preliminaries} we give basic definitions. In
Section~\ref{sec:main} we restate our main results and give some remarks. In the
last two sections we prove the main results.

\medskip

\section{Preliminaries}
\label{sec:preliminaries}

In this section we give necessary definitions to prove the main results.

\subsection{Basic definitions}
\label{sec:basic-definitions}

Denote by $\NN$ the set of nonnegative integers. For $n\in\NN$, we denote
$[n]=\{1,2,\dots,n\}$. If $\alpha\in\NN^n$, the $i$th entry of $\alpha$ is
denoted by $\alpha_i$, i.e., $\alpha=(\alpha_1,\dots,\alpha_n)$. For
$\alpha,\beta\in\NN^n$, we write $\alpha<\beta$ (resp.~$\alpha\le \beta$) if
$\alpha_i<\beta_i$ (resp.~$\alpha_i\le\beta_i$) for all $1\le i\le n$.

An element $\alpha\in\NN^n$ is called a \emph{partition} if
$\alpha_1\ge\dots\ge\alpha_n$. Denote by $\Par_n$ the set of partitions in
$\NN^n$.

Let $\lambda\in\Par_n$. The \emph{Young diagram} of $\lambda$ is the set
$\{(i,j)\in\ZZ\times\ZZ: 1\le i\le n \mbox{ and } 1\le j\le \lambda_i \}$. We
will identify $\lambda$ with its Young diagram. Therefore a partition is
considered as a sequence of nonnegative integers and also as a set of pairs of
positive integers. Each element $(i,j)\in\lambda$ is called a \emph{cell}. The
Young diagram $\lambda$ will be visualized as an array of squares where we place
a square in row $i$ and column $j$ for each $(i,j)\in\lambda$ using the matrix
coordinates. The \emph{transpose} $\lambda'$ of $\lambda$ is the partition given
by $\lambda'=\{(j,i):(i,j)\in\lambda\}$. Note that if
$\lambda=(\lambda_1,\dots,\lambda_n)\in\Par_n$, then
$\lambda'\in \Par_{\lambda_1}$. See Figure~\ref{fig:yd}.

\begin{figure}
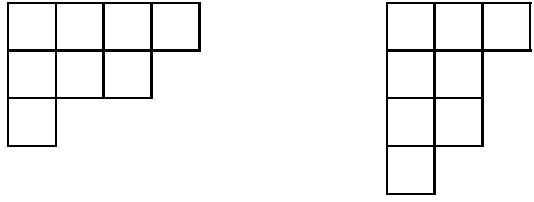

  \centering
  \ydiagram{4,3,1,0}  \qquad \qquad \qquad \ydiagram{3,2,2,1}
  \caption{The Young diagram of $\lambda=(4,3,1)$ on the left and its transpose
  $\lambda'=(3,2,2,1)$ on the right.}
  \label{fig:yd}
\end{figure}

Note that for two partitions $\lambda$ and $\mu$, we have $\mu\subseteq\lambda$
(as Young diagrams) if and only if $\mu\le \lambda$ (as elements in $\NN^n$). We
will mostly use the notation $\mu\subseteq\lambda$ since this emphasizes that
$\mu$ and $\lambda$ are Young diagrams.

For two partitions $\lambda$ and $\mu$ with $\mu\subseteq\lambda$, the
\emph{skew shape} $\lambda/\mu$ is the set-theoretic difference $\lambda-\mu$ of
their Young diagrams. In other words, if $\lambda,\mu\in\Par_n$ satisfy
$\mu\subseteq\lambda$, then $\lm$ is the set of pairs $(i,j)\in\ZZ\times\ZZ$
such that $1\le i\le n$ and $\mu_i+1\le j\le \lambda_i$. A \emph{reverse plane
  partition (RPP)} of shape $\lm$ is a filling of $\lm$ with positive integers
such that the entries are weakly increasing in each row and column. If $R$ is an
RPP of shape $\lm$, the $(i,j)$ entry of $R$ is denoted by $R(i,j)$. The
\emph{transpose} of $R$ is the RPP $R'$ of shape $\lambda'/\mu'$ given by
$R'(i,j)=R(j,i)$ for all $(i,j)\in\lambda'/\mu'$. See Figure~\ref{fig:rpp}.

\begin{figure}
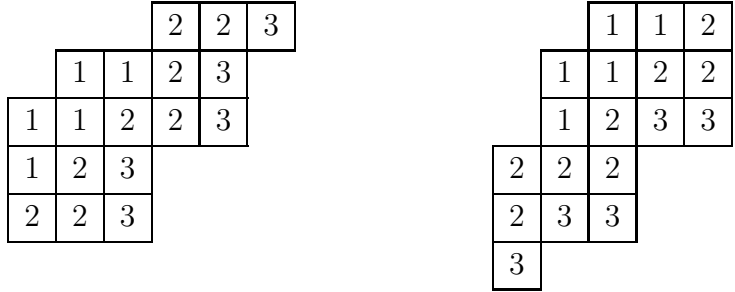

  \centering
\begin{ytableau}
\none & \none & \none & 2 & 2 & 3 \\
\none & 1 & 1 & 2 & 3\\
1 & 1 & 2 & 2 & 3\\
1 & 2 & 3\\
2 & 2 & 3\\
\none
\end{ytableau} \qquad\qquad\qquad
\begin{ytableau}
  \none & \none & 1 & 1 & 2\\
  \none & 1 & 1 & 2 & 2\\
  \none & 1 & 2 & 3 & 3\\
  2 & 2 & 2\\
  2 & 3 & 3\\
  3
\end{ytableau}
  \caption{An RPP $R$ of shape $(6,5,5,3,3)/(3,1)$ on the left and its transpose
    $R'$ on the right.}
  \label{fig:rpp}
\end{figure}

For $\lambda,\mu\in\NN^n$, the set of RPPs of shape $\lm$ is denoted by
$\RPP_{\lm}$. If $\mu\not\in\Par_n$, $\lambda\not\in\Par_n$, or
$\mu\not\subseteq\lambda$, then $\RPP_{\lm}$ is defined to be the empty set. For
$R\in\RPP_{\lm}$, define
\[
\wt(R)=\prod_{i\ge1}x_i^{a_i(R)} t_i^{b_i(R)},
\]
where $a_i(R)$ is the number of columns of $R$ containing an $i$ and $b_i(R)$ is
the number of cells $(i,j)$ such that $(i,j), (i+1,j)\in \lm$ and
$R(i,j)=R(i+1,j)$. For example, if $R$ is the RPP shown in Figure~\ref{fig:rpp}
on the left, then $\wt(R)=x_1^3x_2^5x_3^3t_1t_2^3t_3t_4^2$ and
$\wt(R')=x_1^3x_2^5x_3^5t_1^2t_2t_3t_4$.

Let $x=\{x_1,x_2,\dots\}$ and $t=\{t_1,t_2,\dots\}$ be sets of variables. For
$r\in\NN$ and $s\in\ZZ$, define
\[
X_{(r,s]}=x_{r+1}+x_{r+2}+\dots+x_{s},
\]
where empty sums are zero, i.e., $X_{(r,s]}=0$ if $r\ge s$. In other words,
$X_{(r,s]}$ is the sum of the variables $x_i$ for the integers $i$ in the
interval $(r,s]=\{u\in \mathbb{R}: r<u\le s\}$. We define $T_{(r,s]}$ in the
same way using the variables $t_i$. For integers $i$, we also define
\[
X_i =x_1+x_2+\dots+x_i, \qquad
T_i = t_1+t_2+\dots+t_i,
\]
where empty sums are zero, i.e., $X_i=T_i=0$ if $i\le 0$. Note that if $0\le
r\le s$, then $X_{(r,s]} = X_s-X_r$.

Let $z=\{z_i:i\in I\}$ be a set of variables, where $I\subseteq\NN$. The
\emph{elementary symmetric function} $e_n(z)$ and the \emph{complete homogeneous
  symmetric function} $h_n(z)$, for $n\ge1$, are defined by
\begin{align*}
  e_n(z) &= \sum_{i_1<\dots<i_n \:\:\mathrm{in}\:\: I} z_{i_1} \cdots z_{i_n},\\
  h_n(z) &= \sum_{i_1\le \dots\le i_n \:\:\mathrm{in}\:\: I} z_{i_1} \cdots z_{i_n}.
\end{align*}
We define $e_0(z)=h_0(z)=1$ and $e_k(z)=h_k(z)=0$ for $k<0$.
Note that $e_k(z)=0$ whenever $z=\{z_i:i\in I\}$ satisfies $\left|I\right|<k$.

\subsection{Flagged refined dual stable Grothendieck polynomial}

For $\lambda,\mu\in\NN^n$, the \emph{refined dual stable Grothendieck
  polynomial} $\wg_{\lambda/\mu}(x;t)$ is defined by
\[
  \wg_{\lambda/\mu}(x;t) =
\sum_{R\in \RPP_{\lambda/\mu}} \wt(R).
\]
The polynomials $\wg_\lm(x;t)$ were introduced by Galashin, Grinberg, and Liu
\cite{GGL2016}. They showed that $\wg_\lm(x;t)$ is symmetric in the variables
$x$ but not in the variables $t$.

For $\alpha,\beta,\lambda,\mu\in\NN^n$, define
$\RPP_{\lambda/\mu}^{\row(\alpha,\beta)}$ to be the set of RPPs $R$ of shape
$\lm$ such that $\alpha_i+1\le R(i,j)\le \beta_i$ for all $(i,j)\in\lm$.
Similarly, for $\alpha,\beta\in\NN^n$ and partitions $\lambda,\mu$ with
$\lambda',\mu'\in\NN^n$, define $\RPP_{\lambda/\mu}^{\col(\alpha,\beta)}$ to be
the set of RPPs $R$ of shape $\lm$ such that $\alpha_j+1\le R(i,j)\le \beta_j$
for all $(i,j)\in\lm$.

The \emph{row-flagged refined dual stable Grothendieck polynomial}
$\wg_{\lambda/\mu}^{\row(\alpha,\beta)}(x;t)$ and the \emph{column-flagged refined
  dual stable Grothendieck polynomial}
$\wg_{\lambda/\mu}^{\col(\alpha,\beta)}(x;t)$ are defined by
\begin{align*}
  \wg_{\lambda/\mu}^{\row(\alpha,\beta)}(x;t) &=
\sum_{R\in \RPP_{\lambda/\mu}^{\row(\alpha,\beta)}} \wt(R),\\
  \wg_{\lambda/\mu}^{\col(\alpha,\beta)}(x;t) &=
\sum_{R\in \RPP_{\lambda/\mu}^{\col(\alpha,\beta)}} \wt(R).
\end{align*}
For simplicity we will sometimes omit $(x;t)$ and write
$\wg_{\lambda/\mu}^{\row(\alpha,\beta)}$ and
$\wg_{\lambda/\mu}^{\col(\alpha,\beta)}$.

\subsection{Plethystic substitution}
\label{sec:plethysm}

Let $\Lambda=\Lambda_{\mathbb{Q}}$ denote the ring of symmetric functions with
rational coefficients. The \emph{power sum symmetric functions}
$p_k(x)=x_1^k+x_2^k+\cdots$ generate $\Lambda$ as a $\mathbb{Q}$-algebra. Let
$\mathbb{Q}[[a_1,a_2,\dots]]$ denote the ring of formal power series in
variables $a_1,a_2,\dots$ with rational coefficients. Once
$A\in\mathbb{Q}[[a_1,a_2,\dots]]$ is fixed, the \emph{plethystic substitution}
$f[A]$ for $f\in\Lambda$ is defined by the following rules:
\begin{itemize}
\item for $k\ge1$, $p_k[A]$ is obtained from $A$ by replacing each $a_i$ by
  $a_i^k$,
\item the map $f\mapsto f[A]$ is a ring homomorphism from $\Lambda$ to
  $\mathbb{Q}[[a_1,a_2,\dots]]$.
\end{itemize}

If $A=a_1+\dots+a_n$, then $p_k[A]=a_1^k+\dots+a_n^k=p_k(a_1,\dots,a_n)$, which
implies $f[A]=f(a_1,\dots,a_n)$ for all $f\in\Lambda$. We refer the reader to
\cite{Loehr_2010} for more details on plethystic substitution. We need the
following well known properties of the plethystic substitution.

\begin{prop}\label{prop:plethysm}
  Let $A,B\in \mathbb{Q}[[a_1,a_2,\dots]]$ and $f\in \Lambda$.
  Then 
\begin{align*}
  f[A+B] &=  \sum_{(f)}f_{(1)}[A] f_{(2)}[B],\\
  f[-A] &= (S(f))[A],
\end{align*}
where the Sweedler notation is used and $S$ is the antipode of the Hopf algebra
of symmetric functions. See \cite{grinberg14:hopf_algeb_combin} for more details
on the Sweedler notation and the antipode.
\end{prop}

In this paper we only need to compute $f[A]$ and $f[A-B]$ when $f=e_k$ or
$f=h_k$, and both $A$ and $B$ are sums of variables taken from
$x=(x_1,x_2,\dots)$ and $t=(t_1,t_2,\dots)$. If $A$ and $B$ are any formal power
series with integer coefficients, since $Sh_k=(-1)^ke_k$ and $Se_k=(-1)^kh_k$,
we have
\begin{align}
  \label{eq:h_k[-A]}
  h_k[-A]&= (-1)^k e_k[A],\\
  \label{eq:e_k[-A]}
  e_k[-A]&= (-1)^k h_k[A],\\
  \label{eq:h_k}
  h_k[A-B]&= \sum_{i=0}^kh_{k-i}[A](-1)^i e_i[B],\\
  \label{eq:e_k}
  e_k[A-B]&= \sum_{i=0}^ke_{k-i}[A](-1)^i h_i[B].
\end{align}

\section{Main results}
\label{sec:main}

In this section we restate our main results, Theorem~\ref{thm:main} in the
introduction, as two separate theorems, Theorems~\ref{thm:col_flag2} and
\ref{thm:row_flag}, and prove their corollaries. The main results will be proved
in the next two sections.

The following theorem is Theorem~\ref{thm:col_flag3} in the introduction, which
is equivalent to one of the main results.

\begin{thm}\label{thm:col_flag}
  Let $\alpha,\beta\in\NN^n$ and $\mu,\lambda\in\Par_n$. 
If $\alpha_i\le \alpha_{i+1}+1$ and $\beta_i\le \beta_{i+1}+1$ whenever
$\mu_i<\lambda_{i+1}$, then
\[
\wg_{\lambda'/\mu'}^{\col(\alpha,\beta)}(x;t)  = \det \left(
    e_{\lambda_i-\mu_j-i+j}(x_{\alpha_j+1},\dots, x_{\beta_i},t_{\mu_j+1},\dots,t_{\lambda_i-1})
  \right)_{1\le i,j\le n}.
\]
\end{thm}

The following theorem is the first main result in this paper. Using simple
determinant evaluation techniques we show that Theorem~\ref{thm:col_flag} is
equivalent to this theorem, see Proposition~\ref{prop:equivalence}.

\begin{thm} \label{thm:col_flag2}
  Let $\alpha,\beta\in\NN^n$ and $\mu,\lambda\in\Par_n$. 
If $\alpha_i\le \alpha_{i+1}+1$ and $\beta_i\le \beta_{i+1}+1$ whenever
$\mu_i<\lambda_{i+1}$, then
  \[
  \wg_{\lambda'/\mu'}^{\col(\alpha,\beta)}(x;t) = \det \left(
    e_{\lambda_i-\mu_j-i+j}[X_{(\alpha_j,\beta_i]}+T_{\lambda_i-1}-T_{\mu_j}]
  \right)_{1\le i,j\le n}.
\]
\end{thm}

Note that the $(i,j)$ entry of the matrix in Theorem~\ref{thm:col_flag} 
can be written as
\[
e_{\lambda_i-\mu_j-i+j}(x_{\alpha_j+1},\dots, x_{\beta_i},t_{\mu_j+1},\dots,t_{\lambda_i-1})
= e_{\lambda_i-\mu_j-i+j}[X_{(\alpha_j,\beta_i]}+T_{(\mu_j,\lambda_i-1]}].
\]
If we replace $T_{(\mu_j,\lambda_i-1]}$ by $T_{\lambda_i-1}-T_{\mu_j}$, we
obtain Theorem~\ref{thm:col_flag2}. However, unlike the $t$ variables, we cannot
replace $X_{(\alpha_j,\beta_i]}$ by $X_{\beta_i}-X_{\alpha_j}$.

The following theorem is the second main result, which is a dual version of
Theorem~\ref{thm:col_flag2}.

\begin{thm}\label{thm:row_flag}
  Let $\alpha,\beta\in\NN^n$ and $\mu,\lambda\in\Par_n$. 
If $\alpha_i\le \alpha_{i+1}$ and $\beta_i\le \beta_{i+1}$ whenever
$\mu_i<\lambda_{i+1}$, then
\[
\wg_{\lambda/\mu}^{\row(\alpha,\beta)}(x;t) = \det \left(
      h_{\lambda_i-\mu_j-i+j}[X_{(\alpha_j,\beta_i]}+T_{i-1}-T_{j-1}]
      \right)_{1\le i,j\le n}.
\]
\end{thm}

Theorems~\ref{thm:col_flag2} and~\ref{thm:row_flag} combined yield
Theorem~\ref{thm:main}. Note that, for the variables in the plethystic
substitution in Theorem~\ref{thm:row_flag}, there are four ways of choosing the
$x$ and $t$ variables from $\{X_{(\alpha_j,\beta_i]},X_{\beta_i}-X_{\alpha_j}\}$
and $\{T_{(j-1,i-1]},T_{i-1}-T_{j-1}\}$, respectively. In contrast to
Theorem~\ref{thm:col_flag2}, the choice in Theorem~\ref{thm:row_flag} is the
only one that gives a correct formula.

\begin{remark}\label{rem:wachs}
  Recall that in the formula \eqref{eq:wachs2} the assumption is
  $\alpha_i-\mu_i\le \alpha_{i+1}-\mu_{i+1}+1$ and $\beta_i-\lambda_i\le
  \beta_{i+1}-\lambda_{i+1}+1$ whenever $\mu_i<\lambda_{i+1}$. On the other hand
  the assumption in Theorem~\ref{thm:col_flag2} is $\alpha_i\le \alpha_{i+1}+1$
  and $\beta_i\le \beta_{i+1}+1$ whenever $\mu_i<\lambda_{i+1}$. We cannot replace
  the assumption in Theorem~\ref{thm:col_flag2} by that in \eqref{eq:wachs2}.
  For example, if $\lambda=(3,3)$, $\mu=(2)$, $\alpha=(2,0)$, and $\beta=(2,2)$,
  then $\wg_{\lambda'/\mu'}^{\col(\alpha,\beta)}(x;t)=0$ but
 \[
 \det \left(
    e_{\lambda_i-\mu_j-i+j}[X_{(\alpha_j,\beta_i]}+T_{\lambda_i-1}-T_{\mu_j}]
  \right)_{1\le i,j\le n} = \det
   \begin{pmatrix}
0 & x_1x_2t_1t_2 \\
1 & e_3(x_1,x_2,t_1,t_2)
   \end{pmatrix} \ne 0.
 \]
\end{remark}

\begin{remark}
  In Theorem~\ref{thm:row_flag} the assumption $\alpha_i\le \alpha_{i+1}$ and
  $\beta_i\le \beta_{i+1}$ whenever $\mu_i<\lambda_{i+1}$ is necessary. For
  example, if $\lambda=(2,2)$, $\mu=(1)$, $\alpha=(1,0)$, and $\beta=(1,1)$,
  then $\wg_{\lambda/\mu}^{\row(\alpha,\beta)}=0$ but
 \[
   \det \left(
     h_{\lambda_i-\mu_j-i+j}[X_{(\alpha_j,\beta_i]}+T_{i-1}-T_{j-1}]
      \right)_{1\le i,j\le n} = \det
   \begin{pmatrix}
0 & x_{1}^{3} - x_{1}^{2} t_{1} \\
1 & x_1^2
   \end{pmatrix} \ne 0.
 \]

 Moreover, the assumption $\lambda,\mu\in\Par_n$ is also necessary. If
 $\lambda=(1,1)$, $\mu=(0,1)$, $\alpha=(0,0)$, and $\beta=(1,1)$, then
 $\wg_{\lambda/\mu}^{\row(\alpha,\beta)}=0$ but
 \[
   \det \left(
     h_{\lambda_i-\mu_j-i+j}[X_{(\alpha_j,\beta_i]}+T_{i-1}-T_{j-1}]
      \right)_{1\le i,j\le n} = \det
   \begin{pmatrix}
x_1 & x_{1} - t_{1} \\
1 & 1
   \end{pmatrix} \ne 0.
 \]
\end{remark}

Let $(a^n)$ denote the sequence $(a,a,\dots,a)$ consisting of $n$ $a$'s. If we
set $\alpha=(0^n)$ and $\beta=(b^n)$ and let $b\to\infty$ in
Theorem~\ref{thm:row_flag}, we obtain the formula \eqref{eq:jt2} in the
introduction, which we state again.

\begin{cor}\label{cor:row_flag}
  For any $\lambda,\mu\in\Par_n$, we have
\[
\wg_{\lambda/\mu}(x;t)= \det \left(
      h_{\lambda_i-\mu_j-i+j}[X+T_{i-1}-T_{j-1}]
      \right)_{1\le i,j\le n}.
\]
\end{cor}
\begin{proof}
  By definition of $\wg_{\lambda/\mu}^{\row(\alpha,\beta)}(x;t)$ and
  Theorem~\ref{thm:row_flag},
\begin{align*}
  \wg_{\lambda/\mu}(x;t)
  &= \lim_{b\to\infty}\wg_{\lm}((x_1,\dots,x_b,0,0,\dots);t)\\
&= \lim_{b\to\infty}\wg^{\row((0^n),(b^n))}_{\lm}(x;t)\\
  &=\lim_{b\to\infty} \det \left(
      h_{\lambda_i-\mu_j-i+j}[X_{(0,b]}+T_{i-1}-T_{j-1}]
    \right)_{1\le i,j\le n}\\
  &=\det \left(
      h_{\lambda_i-\mu_j-i+j}[X+T_{i-1}-T_{j-1}]
      \right)_{1\le i,j\le n}.
\end{align*}
\end{proof}

Similarly, Theorem~\ref{thm:JT} follows from Theorem~\ref{thm:col_flag}, and
\eqref{eq:jt1} follows from \eqref{eq:main1}.

Theorem~\ref{thm:JT2} is the special case $t_i=1$ of
Corollary~\ref{cor:row_flag}. Amanov and Yeliussizov \cite{AmanovYeliussizov}
showed Theorem~\ref{thm:JT2} using Theorem~\ref{thm:JT} and an involution
$\tau:\Lambda\to\Lambda$ satisfying $\tau(\wg_{\lm}(x;t)) =
\wg_{\lambda'/\mu'}(x;t)$ when all $t_i$ are equal to $1$. Unfortunately there
is no such map for the general $t$. To see this suppose that there were an
algebra homomorphism $\psi:\Lambda\to\Lambda$ satisfying $\psi(\wg_{\lm}(x;t)) =
\wg_{\lambda'/\mu'}(x;t)$. Then it must satisfy
  \[
\psi(h_k(x))=\psi(\wg_{(k)}(x;t)) = \wg_{(1^k)}(x;t) = e_k[X+T_{k-1}]. 
  \]
  Since $h_k(x)=\wg_{(k)}(x;t)=\wg_{(k+1)/(1)}(x;t)$, we must also have
  \[
\psi(h_k(x))=\psi(\wg_{(k+1)/(1)}(x;t)) = \wg_{(1^{k+1})/(1)}(x;t) = e_k[X+T_{k}-T_1]. 
  \]
Since $e_k[X+T_{k-1}]\ne e_k[X+T_{k}-T_1]$, the map $\psi$ cannot exist.

\section{A proof of the Jacobi--Trudi formula for $\wg_{\lambda'/\mu'}^{\col(\alpha,\beta)}$}
\label{sec:proof-jacobi-trudi1}

In this section we prove the Jacobi--Trudi formula for
$\wg_{\lambda'/\mu'}^{\col(\alpha,\beta)}$ in Theorem~\ref{thm:col_flag2}. The
basic idea of the proof is to show that both sides of the equation satisfy the
same recurrence relation. We first introduce several definitions.

A \emph{diagram} is just a (finite) set of pairs $(i,j)$ of positive integers.
Similarly to Young diagrams we also visualize a diagram $\rho$ as an array of
squares where a square is placed in row $i$ and column $j$ for each
$(i,j)\in\rho$. For a diagram $\rho$, define $\rho_i$ to be the number of $j$'s
such that $(i,j)\in\rho$. The $k$th \emph{row} (resp.~\emph{column}) of $\rho$
is the set of cells $(i,j)\in\rho$ with $i=k$ (resp.~\textbf{$j=k$}). For two
diagrams $\sigma$ and $\rho$ with $\rho\subseteq\sigma$, denote by $\sigma-\rho$
their set-theoretic difference, which is also a diagram. If a diagram $\rho$ is
a Young diagram with at most $n$ rows, it is identified with the partition
$(\rho_1,\dots,\rho_n)$ as before. If $R$ is an RPP of shape $\lm$ and
$\rho\subseteq\lm$ is a diagram, the restriction of $R$ to $\rho$ is denoted by
$R|_\rho$. We extend the definition of an RPP of shape $\lm$ to an RPP of shape
$\rho$ for any diagram $\rho$ as follows. A \emph{reverse plane partition (RPP)}
of shape $\rho$ is a filling $R$ of $\rho$ with positive integers such that
$R(i,j)\le R(i',j')$ for all $(i,j),(i',j')\in\rho$ with $i\le i'$ and $j\le
j'$. The notation used for RPPs of shape $\lm$ will be extended to RPPs of shape
$\rho$ in the obvious way. For example, $\RPP_{\rho}$ is the set of RPPs of
shape $\rho$ and $\RPP^{\row(\alpha,\beta)}_{\rho}$ is the set of elements
$R\in\RPP_{\rho}$ with the additional condition that $\alpha_i+1\le R(i,j)\le
\beta_i$ for all $(i,j)\in\rho$.

Let $\mu,\lambda\in\Par_n$ with $\mu\subseteq\lambda$. We define a total order
$\prec$ on the cells in $\lm$ as follows: $(i,j) \prec (i',j')$ if and only if
$j>j'$ or $j=j'$ and $i<i'$. Note that by definition $(i,j) \prec (i',j')$
implies $(i,j)\ne (i',j')$. Denote by $(\lm)^{(m)}$ the set of the first $m$
cells in $\lm$ in the total order $\prec$. Note that $(\lm)^{(m)}$ is not
necessarily a skew shape, see Figure~\ref{fig:rpp2}.

\begin{defn}\label{defn:R0}
  Let $\alpha,\beta\in\NN^n$ and $\mu,\lambda\in\Par_n$ with
  $\mu\subseteq\lambda$. Let $\rho=(\lm)^{(m)}$ for some $0\le m\le|\lm|$ and
  let $R_0\in \RPP^{\row(\alpha,\beta)}_{\rho}$. Then we define
\begin{align*}
  \RPP^{\row(\alpha,\beta)}_{\lm}(R_0) 
  &=\{R\in \RPP^{\row(\alpha,\beta)}_{\lm}: R|_\rho = R_0\},\\
 C(\rho)&=\{1\le i\le n: \rho_i>0\},\\
B(R_0,\beta)&=(\wb_1,\dots,\wb_n),
\end{align*}
where $\wb_i$ is defined by
\[
    \wb_i=
\begin{cases}
  R_0(i,\lambda_i-\rho_i+1), & \mbox{if $i\in C(\rho)$},\\
  \beta_i, & \mbox{if $i\notin C(\rho)$}.
\end{cases}
\]
Note that if $i\in C(\rho)$, then $(i,\lambda_i-\rho_i+1)$ is the leftmost cell
in the $i$th row of $\rho$.
\end{defn}

One may consider an element in $\RPP^{\row(\alpha,\beta)}_{\lm}(R_0)$ as an RPP
in $\RPP^{\row(\alpha,\beta)}_\lm$ that can be obtained from $R_0$ by filling
the remaining cells in $(\lm) - \rho$. The motivation for introducing
$\RPP^{\row(\alpha,\beta)}_{\lm}(R_0)$ is to construct an RPP of shape $\lm$ by
filling the cells in $\lm$ one at a time with respect to the order of the cells
given by $\prec$. This will allow us to find a recurrence relation for a
generating function for restricted RPPs.

Note that each element $\wb_i$ in $B(R_0,\beta)$ acts as an upper bound for the
remaining entries in row $i$ for an RPP in
$\RPP^{\row(\alpha,\beta)}_{\lm}(R_0)$. For example, if $R_0$ is the RPP shown
in Figure~\ref{fig:rpp2}, and $\alpha=(0,0,1,1,2)$ and $\beta=(5,5,6,7,7)$, then
$C(\rho)=\{1,2,3,4\}$ and $B(R_0,\beta)=(3,1,3,4,7)$. Note also that if the
$i$th row of $(\lm)-\rho$ is empty, then the lower and upper bounds for the
entries in row $i$ are irrelevant.

\begin{figure}
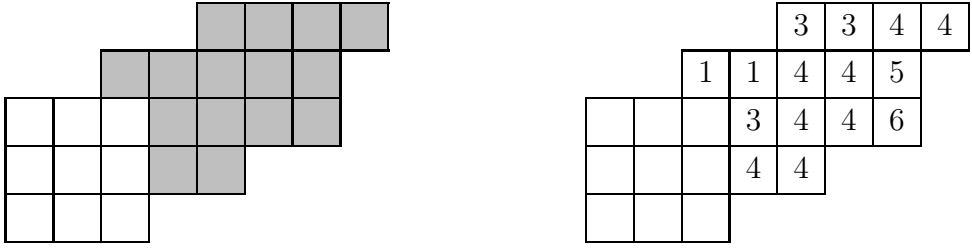

  \centering
\begin{ytableau}
\none & \none & \none &\none  & *(lightgray)  & *(lightgray) & *(lightgray) & *(lightgray)\\
\none & \none & *(lightgray) & *(lightgray) & *(lightgray) &  *(lightgray)& *(lightgray)\\
 &  &  & *(lightgray) & *(lightgray)&  *(lightgray)& *(lightgray)\\
 &  &  & *(lightgray) &*(lightgray) \\
 &  & \\
\none
\end{ytableau} \qquad\qquad\qquad
\begin{ytableau}
\none & \none & \none & \none &  3 & 3 & 4 & 4\\
\none & \none & 1 & 1 & 4 & 4 & 5\\
 &  &  & 3 &4 & 4 & 6\\
 &  &  & 4 & 4 \\
 &  & \\
\none
\end{ytableau}
\caption{The left diagram shows $\lm$ and $\rho=(\lm)^{(m)}$, where
  $\lambda=(8,7,7,5,3)$, $\mu=(4,2)$, $m=15$, and the cells in $\rho$ are the
  gray cells. The right diagram shows an RPP of shape $\rho$.}
  \label{fig:rpp2}
\end{figure}

For $R\in\RPP_{\lm}$, we define 
\[
\wt(R)=\prod_{i\ge1}x_i^{a_i(R)} t_i^{b_i(R)},
\]
where $a_i(R)$ is the number of columns containing an $i$ and $b_i(R)$ is the
number of cells $(i,j)$ such that $(i,j),(i+1,j)\in\lm$ and $R(i,j)=R(i+1,j)$.
We also define
\[
\owt(R)=\prod_{j\ge1}x_j^{\overline{a}_j(R)} t_j^{\overline{b}_j(R)},
\]
where $\overline{a}_j(R)$ is the number of rows containing a $j$ and
$\overline{b}_j(R)$ is the number of cells $(i,j)$ such that
$(i,j),(i,j+1)\in\lm$ and $R(i,j)=R(i,j+1)$.

Note that an RPP $R$ and its transpose $R'$ satisfy
$\owt(R)=\wt(R')$.

\begin{defn}\label{defn:R}
  For any $\alpha,\beta\in\NN^n$, $\mu,\lambda\in\Par_n$ with
  $\mu\subseteq\lambda$, and a fixed RPP $R_0$ of shape $\rho=(\lm)^{(m)}$,
  define
\begin{align*}
  \R^{\alpha,\beta}_{\lambda,\mu}(R_0) 
    &=\sum_{R\in \RPP^{\row(\alpha,\beta)}_{\lm}(R_0)} \wt(R),\\
  \oR^{\alpha,\beta}_{\lambda,\mu}(R_0) 
    &=\sum_{R\in \RPP^{\row(\alpha,\beta)}_{\lm}(R_0)} \owt(R).
\end{align*}
\end{defn}

Note that, by definition, 
\begin{align*}
  \wg_{\lambda/\mu}^{\row(\alpha,\beta)}(x;t) &= \R^{\alpha,\beta}_{\lambda,\mu}(\emptyset),\\
  \wg_{\lambda/\mu}^{\col(\alpha,\beta)}(x;t) &=\oR^{\alpha,\beta}_{\lambda',\mu'}(\emptyset),
\end{align*}
where $\emptyset$ is the unique filling of the empty diagram and we define
$\R^{\alpha,\beta}_{\lambda,\mu}(\emptyset)=\oR^{\alpha,\beta}_{\lambda,\mu}(\emptyset)=0$
if $\mu\not\subseteq\lambda$. In order to avoid using transposes in the proof,
instead of the latter equation above we will consider
\[
  \wg_{\lambda'/\mu'}^{\col(\alpha,\beta)}(x;t)
  =\oR^{\alpha,\beta}_{\lambda,\mu}(\emptyset).
\]

Definitions~\ref{defn:R0} and \ref{defn:R} will also be used in the next section.
We need one more definition for this section.

\begin{defn}\label{defn:E}
For $\alpha,\beta,\mu,\lambda\in\NN^n,C\subseteq[n]$, and $1\le i,j\le n$,
define
\begin{align*}
  e_{\lambda,\mu}^{\alpha,\beta}(i,j) &=
  e_{\lambda_i-i-\mu_j+j}[X_{(\alpha_j,\beta_i]}+T_{\lambda_i-1}-T_{\mu_j}],\\
  \ove_{\lambda,\mu}^{\alpha,\beta}(i,j) &=
  e_{\lambda_i-i-\mu_j+j}[X_{(\alpha_j,\beta_i-1]}+T_{\lambda_i}-T_{\mu_j}],\\
  e_{\lambda,\mu}^{\alpha,\beta}(C;i,j) &=
  \begin{cases}
   e_{\lambda,\mu}^{\alpha,\beta}(i,j), & \mbox{if $i\not\in C$},\\
   \ove_{\lambda,\mu}^{\alpha,\beta}(i,j), & \mbox{if $i\in C$},
  \end{cases}\\
  E_{\lambda,\mu}^{\alpha,\beta}(C) &=
  \det(e_{\lambda,\mu}^{\alpha,\beta}(C;i,j))_{1\le i,j\le n},\\
  E_{\lambda,\mu}^{\alpha,\beta} &=
  \det(e_{\lambda,\mu}^{\alpha,\beta}(i,j))_{1\le i,j\le n}.
\end{align*}
If $\mu\subseteq\lambda$, $\rho=(\lm)^{(m)}$, and
$R_0\in\RPP^{\row(\alpha,\beta)}_\rho$, we define
\begin{equation}
  \label{eq:defEE}
  \EE_{\lambda,\mu}^{\alpha,\beta}(R_0) = \owt(R_0) 
  E_{\lambda-\rho,\mu}^{\alpha,B(R_0,\beta)}(C(\rho)).
\end{equation}
\end{defn}
Note that
\[
E_{\lambda,\mu}^{\alpha,\beta}=E_{\lambda,\mu}^{\alpha,\beta}(\emptyset)
=\EE_{\lambda,\mu}^{\alpha,\beta}(\emptyset),
\]
where the second $\emptyset$ stands for the unique filling of 
the empty diagram $(\lm)^{(0)}$.

Using the notation above Theorem~\ref{thm:col_flag2} can be stated as
\[
E_{\lambda,\mu}^{\alpha,\beta} = \oR_{\lambda,\mu}^{\alpha,\beta}(\emptyset).
\]
Our strategy is to show that both sides of the above equation satisfy the same
recurrence relation.

We will frequently use the following lemmas, which can easily be proved using
elementary linear algebra.

\begin{lem}\label{lem:det=0}
  Let $A=(a_{i,j})_{1\le i,j\le n}$ be a matrix. If there is an integer $1\le
  k\le n$ such that $a_{i,j}=0$ for all $k\le i\le n$ and $1\le j\le k$, then
  $\det(A)=0$. Similarly, if there is an integer $1\le k\le n$ such that
  $a_{i,j}=0$ for all $1\le i\le k$ and $k\le j\le n$, then $\det(A)=0$.
\end{lem}

\begin{lem}\label{lem:det=detdet}
  Let $A=(a_{i,j})_{1\le i,j\le n}$ be a matrix. If there is an integer $0\le
  k\le n$ such that $a_{i,j}=0$ for all $k+1\le i\le n$ and $1\le j\le k$, then
  $\det(A)= \det (a_{i,j})_{1\le i,j\le k} \det (a_{i,j})_{k+1\le i,j\le n}$.
  Similarly, if there is an integer $0\le k\le n$ such that
  $a_{i,j}=0$ for all $1\le i\le k$ and $k+1\le j\le n$, then
  $\det(A)= \det (a_{i,j})_{1\le i,j\le k} \det (a_{i,j})_{k+1\le i,j\le n}$.
\end{lem}

\begin{lem}\label{lem:det=det1det}
  Let $A=(a_{i,j})_{1\le i,j\le n}$ be a matrix. Assume that there is
  an integer $1 \le k\le n$ such that $a_{i,j}= \chi(i=j=k)$ for all
  $k\le i\le n$ and $1\le j\le k$. Then,
  $\det(A)= \det (a_{i,j})_{1\le i,j\le k-1} \det (a_{i,j})_{k+1\le i,j\le n}$.
  Furthermore, each nonzero term in the expansion of $\det(A)$
  must contain the $(k,k)$ entry (which is $1$).
\end{lem}

\subsection{Technical lemmas}
\label{sec:technical-lemmas}

In this subsection we give a list of lemmas that will be used to prove
Theorem~\ref{thm:col_flag2}.

From now on, once $n$ is given, let $\epsilon_k=(0,\dots,0,1,0,\dots,0)$ be the
sequence of $n-1$ zeros and one $1$, where the unique $1$ is at position $k$.
For a statement $p$, we define $\chi(p)=1$ if $p$ is true and $\chi(p)=0$
otherwise. 

Let $\QPar_n$ denote the set of $\alpha\in\NN^n$ such that $\alpha_i\le
\alpha_{i+1}+1$ for all $i\in[n-1]$. Note that if $\alpha\in\QPar_n$ and $1\le
i\le j\le n$, then $\alpha_i\le \alpha_j +j-i$.

We note a simple but crucial fact:
If $k \in \ZZ$, and if $A$ is a sum of fewer than $k$ variables,
then
\begin{equation}
\label{eq:ekA=0}
e_k[A] = 0.
\end{equation}
(This is because $e_k[A] = e_k\left(a_1, a_2, \ldots, a_m\right)$
whenever $A = a_1 + a_2 + \cdots + a_m$ is a sum of variables.)

Now we give a list of lemmas.

\begin{lem}\label{lem:e_k[Z]}
  Let $Z$ be a formal power series with integer coefficients and let $z$ be any
  (single) variable. Then, for any integer $k$,
  \[
e_k[Z] = e_k[Z-z] + z e_{k-1}[Z-z].
  \]
\end{lem}
\begin{proof}
  Since $e_k(x)=0$ for $k<0$ and $e_0(x)=1$, the equation is clear when $k\le
  0$. Now we assume $k\ge1$. By \eqref{eq:e_k}, for any $m\ge0$, we have
  \[
e_m[Z-z] = \sum_{i=0}^m e_{m-i}[Z](-1)^ih_i(z)=\sum_{i=0}^m e_{m-i}[Z](-z)^i.
  \]
  Thus
  \[
    e_k[Z-z] + z e_{k-1}[Z-z] = \sum_{i=0}^k e_{k-i}[Z](-z)^i
    -\sum_{i=0}^{k-1} e_{k-i-1}[Z](-z)^{i+1} = e_k[Z],
  \]
  as desired.
\end{proof}

\begin{lem}\label{lem:e_k=0}
If $i\ge j\ge0$ and $k>i-j$, then 
\[
e_k[T_i-T_j] = 0.
\]
\end{lem}
\begin{proof}
  By definition we have $e_k[T_i-T_j] = e_k[t_{j+1}+\dots+t_i] =
  e_k(t_{j+1},\dots,t_i)$, which is equal to $0$ because $k > i-j$.
\end{proof}

\begin{lem}\label{lem:initial E}
  Let $\alpha,\beta\in\NN^n$, $\mu\in\Par_n$, and $C\subseteq[n]$. Then
\[
  E^{\alpha,\beta}_{\mu,\mu}(C) = 1.
\] 
\end{lem}
\begin{proof}
  If $i\ge j$, we have $\mu_i-i-\mu_j+j\le 0$, where the equality holds if and
  only if $i=j$. This shows that for all $1\le j\le i\le n$, we have
  $e^{\alpha,\beta}_{\mu,\mu}(C;i,j)= \chi(i=j)$ since
  $e_{\mu_i-i-\mu_j+j}[Y-Z]=\chi(i=j)$ for any sums $Y$ and $Z$ of variables.
  Therefore the matrix $(e^{\alpha,\beta}_{\mu,\mu}(C;i,j))_{1\le i,j\le n}$ is
  upper uni-triangular and $E^{\alpha,\beta}_{\mu,\mu}(C) =
  \det(e^{\alpha,\beta}_{\mu,\mu}(C;i,j))_{1\le i,j\le n}=1$.
\end{proof}

\begin{lem}\label{lem:E(C)=0}
  Let $\alpha,\beta\in\NN^n$ and $\lambda,\mu\in\Par_n$ with
  $\mu\not\subseteq\lambda$. Then for any subset $C\subseteq[n]$,
\[
  E^{\alpha,\beta}_{\lambda,\mu}(C) = 0.
\] 
\end{lem}
\begin{proof}
  Since $\mu\not\subseteq\lambda$, there is an integer $1\le k\le n$ such that
  $\mu_k>\lambda_k$. Then for all $k\le i\le n$ and $1\le j\le k$, we have
  $\lambda_i-i-\mu_j+j\le \lambda_k-k-\mu_k+k<0$, and therefore
  $e^{\alpha,\beta}_{\lambda,\mu}(C;i,j)=0$ since $e_{m}(x)=0$ for $m<0$. By
  Lemma~\ref{lem:det=0} this shows that $E^{\alpha,\beta}_{\lambda,\mu}(C) =
  \det(e^{\alpha,\beta}_{\lambda,\mu}(C;i,j))_{1\le i,j\le n}=0$.
\end{proof}

\begin{lem}\label{lem:E=0}
  Let $\alpha,\beta\in\QPar_n$ and $\lambda,\mu\in\Par_n$. Suppose that
  $\alpha_k\ge \beta_k$ and $\mu_k<\lambda_k$ for some $1\le k\le n$. Then
\[
  E^{\alpha,\beta}_{\lambda,\mu} = 0.
\] 
\end{lem}
\begin{proof}
  By Lemma~\ref{lem:det=0} it is enough to show that
  $e^{\alpha,\beta}_{\lambda,\mu}(i,j)=0$ assuming $1\le i\le k$ and $k\le j\le
  n$. Since $\alpha,\beta\in\QPar_n$, we have $\alpha_j
  +j-k\ge\alpha_k\ge\beta_k\ge\beta_i-k+i$. Thus $\beta_i-\alpha_j\le j-i$ and
  $X_{(\alpha_j,\beta_i]}$ is a sum of at most $j-i$ variables. Furthermore,
  $T_{\lambda_i-1}-T_{\mu_j}$ is a sum of $\lambda_i - 1 - \mu_j$ variables
  (since $\lambda_i\ge\lambda_k>\mu_k\ge\mu_j$). Using these two facts, we
  obtain
  \[
    e^{\alpha,\beta}_{\lambda,\mu}(i,j)
    =e_{\lambda_i-i-\mu_j+j}[X_{(\alpha_j,\beta_i]} + T_{\lambda_i-1}-T_{\mu_j}]
    = 0,
  \]
  because $X_{(\alpha_j,\beta_i]} + T_{\lambda_i-1}-T_{\mu_j}$ is a sum of at
  most $\lambda_i-i-\mu_j+j-1$ variables
  (and because of \eqref{eq:ekA=0}). This completes the proof.
\end{proof}

\begin{lem}\label{lem:rec e}
  Let $\alpha\in\QPar_n$, $\beta\in\NN^n$, $\mu\in\Par_n$, and
  $\lambda\in\NN^n$. Suppose that $k\in[n]$ is an integer satisfying the
  following conditions:
  \begin{enumerate}
  \item $\mu_k<\lambda_k$,
  \item $\alpha_k<\beta_k$, and
  \item if $1\le j<k$ and $\mu_j<\lambda_k$, then $\alpha_j<\beta_k$.
  \end{enumerate}
   Then for any $C\subseteq[n]$, we have
\[
  E_{\lambda,\mu}^{\alpha,\beta}(C)
  =E_{\lambda,\mu}^{\alpha,\beta-\epsilon_k}(C\setminus\{k\})
  +y E_{\lambda-\epsilon_k,\mu}^{\alpha,\beta}(C\cup\{k\}),
\]
where
\[
y =
\begin{cases}
  x_{\beta_k}, &\mbox{if $k\not\in C$,}\\
  t_{\lambda_k}, &\mbox{if $k\in C$.}
\end{cases}
\]
\end{lem}
\begin{proof}
We consider the two cases $k\notin C$ and $k\in C$.

\textbf{Case 1:} $k\not\in C$. We claim that, for all $1\le j\le n$,
\begin{equation}\label{eq:e=e+xe}
  e_{\lambda,\mu}^{\alpha,\beta}(k,j)
  =e_{\lambda,\mu}^{\alpha,\beta-\epsilon_k}(k,j)
  +x_{\beta_k} \ove_{\lambda-\epsilon_k,\mu}^{\alpha,\beta}(k,j).
\end{equation}
By definition we can rewrite \eqref{eq:e=e+xe} as 
\begin{multline}\label{eq:7.4}
\quad  e_{\lambda_k-k-\mu_j+j}[X_{(\alpha_j,\beta_k]}+T_{\lambda_k-1}-T_{\mu_j}]
  = e_{\lambda_k-k-\mu_j+j}[X_{(\alpha_j,\beta_k-1]}+T_{\lambda_k-1}-T_{\mu_j}]\\
  + x_{\beta_k}
  e_{\lambda_k-k-\mu_j+j-1}[X_{(\alpha_j,\beta_k-1]}+T_{\lambda_k-1}-T_{\mu_j}].
\end{multline}
If $\alpha_j<\beta_k$, \eqref{eq:7.4} follows from Lemma~\ref{lem:e_k[Z]} with
$Z=X_{(\alpha_j,\beta_k]}+T_{\lambda_k-1}-T_{\mu_j}$ and $z=x_{\beta_k}$.

Suppose now that $\alpha_j\ge \beta_k$. Then $X_{(\alpha_j, \beta_k]} = 0$ and
$X_{(\alpha_j, \beta_k-1]} = 0$. Therefore in order to prove the
claim~\eqref{eq:7.4}, it suffices to show
\begin{equation}
  \label{eq:10}
 e_{\lambda_k-k-\mu_j+j-1}[T_{\lambda_k-1}-T_{\mu_j}] = 0, 
\end{equation}
because then both sides of \eqref{eq:7.4} are equal to
$e_{\lambda_k-k-\mu_j+j}[T_{\lambda_k-1}-T_{\mu_j}]$. We will prove
\eqref{eq:10} by considering the two cases $j\le k$ and $k<j$.

First we assume $j\le k$. Since $\alpha_j\ge \beta_k$ and $\alpha_k<\beta_k$, we
have $j\ne k$, thus $j<k$. Thus, by condition (3), if $\mu_j<\lambda_k$, then
$\alpha_j<\beta_k$, which contradicts the assumption $\alpha_j\ge \beta_k$. Thus
we must have $\lambda_k\le \mu_j$. Since $j\le k$ and $\lambda_k\le \mu_j$, we
have $\lambda_k-k-\mu_j+j-1 \le -1$, which shows \eqref{eq:10}.

Now we assume $k<j$. Then $\lambda_k>\mu_k\ge\mu_j$. Since
$\lambda_k-k-\mu_j+j-1>\lambda_k-\mu_j-1$ and $\lambda_{k}-1\ge \mu_j$, we
obtain \eqref{eq:10} by Lemma~\ref{lem:e_k=0}. This establishes the claim
\eqref{eq:e=e+xe}.

Using \eqref{eq:e=e+xe} and the linearity of the determinant in its $k$th row,
we obtain the identity in the lemma in this case.

\textbf{Case 2:} $k\in C$. We claim that, for all $1\le j\le n$,
\begin{equation}\label{eq:e=e+te}
  \ove_{\lambda,\mu}^{\alpha,\beta}(k,j)
  =e_{\lambda,\mu}^{\alpha,\beta-\epsilon_k}(k,j)
  +t_{\lambda_k} \ove_{\lambda-\epsilon_k,\mu}^{\alpha,\beta}(k,j).
\end{equation}
  By Lemma~\ref{lem:e_k[Z]} with
  $Z=X_{(\alpha_j,\beta_k-1]}+T_{\lambda_k}-T_{\mu_j}$ and $z=t_{\lambda_k}$,
  we have
  \begin{multline*}
    e_{\lambda_k-k-\mu_j+j}[X_{(\alpha_j,\beta_k-1]}+T_{\lambda_k}-T_{\mu_j}]
    = e_{\lambda_k-k-\mu_j+j}[X_{(\alpha_j,\beta_k-1]}+T_{\lambda_k-1}-T_{\mu_j}]
    \\ + t_{\lambda_k} 
    e_{\lambda_k-k-\mu_j+j-1}[X_{(\alpha_j,\beta_k-1]}+T_{\lambda_k-1}-T_{\mu_j}],
  \end{multline*}
  which is exactly \eqref{eq:e=e+te}. Using \eqref{eq:e=e+te} and the linearity
  of the determinant in its $k$th row, we obtain the identity in the lemma in
  this case, which completes the proof.
\end{proof}

\begin{lem}\label{lem:equal rows}
  Suppose that $\alpha, \beta,\mu,\lambda \in \NN^n$, $C \subseteq [n]$, and
  $r \in [n]$ satisfy $r \notin C$, $r-1 \in C$, $\lambda_r = \lambda_{r-1}+1$,
  and $\beta_r = \beta_{r-1}-1$. Then $E_{\lambda, \mu}^{\alpha, \beta}(C) = 0$.
\end{lem}
\begin{proof}
We
  compare rows $r$ and $r-1$ of the matrix in the definition of
\[
  E_{\lambda, \mu}^{\alpha, \beta}(C) =
\det\left(e^{\alpha,\beta}_{\lambda,\mu} (C;i,j)\right)_{1\le i,j\le n}.
\] 
Since $r\notin C$, the $(r,j)$ entry of the matrix is 
\[
  e^{\alpha,\beta}_{\lambda,\mu}(C;r,j)
  =e^{\alpha,\beta}_{\lambda,\mu}(r,j) 
    =e_{\lambda_r-r-\mu_j+j}[X_{(\alpha_j,\beta_r]}+T_{\lambda_r-1}-T_{\mu_j}].
\]
Since $r-1\in C$, the $(r-1,j)$ entry of the matrix is 
\[
  e^{\alpha,\beta}_{\lambda,\mu}(C;r-1,j)
  =\ove^{\alpha,\beta}_{\lambda,\mu}(r-1,j) 
    =e_{\lambda_{r-1}-(r-1)-\mu_j+j}[X_{(\alpha_j,\beta_{r-1}-1]}+T_{\lambda_{r-1}}-T_{\mu_j}].
\]
Since $\lambda_r = \lambda_{r-1}+1$ and $\beta_r = \beta_{r-1}-1$, the right
hand sides of the above two equations are equal. Therefore rows $r-1$ and $r$ of
the matrix are identical, which implies $E_{\lambda, \mu}^{\alpha, \beta}(C) =
0$.
\end{proof}

\subsection{Proof of Theorem~\ref{thm:col_flag2}}

We first show that $\oR^{\alpha,\beta}_{\lambda,\mu}(R_0)$ and
$\EE^{\alpha,\beta}_{\lambda,\mu}(R_0)$ satisfy the same recurrence relation 
under certain conditions.

\begin{prop}\label{prop:rec RPP R0}
  Let $\alpha,\beta\in\QPar_n$ and $\lambda,\mu\in\Par_n$ with $\alpha<\beta$ and
  $\mu<\lambda$. Fix $(r,c)\in\lm$ and $R_0\in
  \RPP^{\row(\alpha,\beta)}_{\rho}$, where $\rho$ is the set of cells
  $(i,j)\in\lm$ with $(i,j)\prec(r,c)$. Let
  $\wb=(\wb_1,\dots,\wb_n)=B(R_0,\beta)$. Then
  \begin{align}
    \label{eq:rec M}
    \oR^{\alpha,\beta}_{\lambda,\mu}(R_0) 
    &= \sum_{k=a}^{\wb_r} \oR^{\alpha,\beta}_{\lambda,\mu}(R_0\cup\{k\}),\\
    \label{eq:rec E}
    \EE^{\alpha,\beta}_{\lambda,\mu}(R_0)
    &= \sum_{k=a}^{\wb_r} \EE^{\alpha,\beta}_{\lambda,\mu}(R_0\cup\{k\}),
  \end{align}
  where $R_0\cup\{k\}$ is the RPP obtained from $R_0$ by adding the cell $(r,c)$
  with entry $k$, and 
  \[
a = \begin{cases}
  \wb_{r-1}, &\mbox{if $(r-1,c)\in\rho$ and $\wb_{r-1}\ge\alpha_r+1$},\\
 \alpha_r+1, &\mbox{otherwise}.
\end{cases}
\]
\end{prop}
\begin{proof}
  Observe that in order to construct an RPP $R$ in
  $\RPP^{\row(\alpha,\beta)}_\lm$ such that $R|_\rho=R_0$ we must fill the
  $(r,c)$ cell of $R$ with one of the integers $a,a+1,\dots,\wb_r$. Therefore
  the first identity \eqref{eq:rec M} is immediate from the definition of
  $\oR^{\alpha,\beta}_{\lambda,\mu}(R_0)$.

  It remains to prove the second identity \eqref{eq:rec E}.
  From the equality \eqref{eq:8a} below, we have $\alpha_r < \wb_r$.

  Let $C=C(\rho)=\{1\le i\le n:
  \rho_i>0\}$ and
    \begin{equation}\label{eq:s}
s = \begin{cases}
  \wb_r-\wb_{r-1}+1, & \mbox{if $(r-1,c)\in\rho$
  and $\wb_{r-1}\ge\alpha_r+1$} ,\\
\wb_r-\alpha_r, &\mbox{otherwise} . 
\end{cases}
  \end{equation}
  Lemma~\ref{lem:s-1} below (specifically, \eqref{eq:s>0})
  shows that $s\ge0$. Since $s = \wb_r - a + 1$, this
  shows that $s$ is the number of integers $k$ satisfying $a\le k\le \wb_r$. We
  will consider the two cases $s=0$ and $s\ge1$.

  We first consider the case $s=0$. In this case, $\wb_r = a-1$,
  so that the right hand side of
  \eqref{eq:rec E} is zero. Recall from \eqref{eq:defEE} that the left
  hand side of \eqref{eq:rec E} is
  \[
    \EE_{\lambda,\mu}^{\alpha,\beta}(R_0) = \owt(R_0)
    E_{\lambda-\rho,\mu}^{\alpha,B(R_0,\beta)}(C(\rho)) = \owt(R_0)
    E_{\lambda-\rho,\mu}^{\alpha,\wb}(C),
\]
since $\wb=B(R_0,\beta)$ and $C = C(\rho)$.
Hence, to prove \eqref{eq:rec E}, it suffices to show that
  \begin{equation}
    \label{eq:14}
  E_{\lambda-\rho,\mu}^{\alpha,\wb}(C) = 0.
  \end{equation}
  From \eqref{eq:s=0-cond} below, we get $r \notin C(\rho) = C$,
  whence $C=C\setminus\{r\}$.
Therefore \eqref{eq:14} follows from Lemma~\ref{lem:s}.
This shows \eqref{eq:rec E} for the case $s=0$.

It remains to prove \eqref{eq:rec E} for the case $s\ge1$. Let
\[
y =
\begin{cases}
  x_{\wb_r}, &\mbox{if $r\not\in C$,}\\
  t_{\lambda_r-\rho_r}, &\mbox{if $r\in C$.}
\end{cases}
\]
Since $(r,c)\in(\lm)-\rho$, we have $\mu_r<(\lambda-\rho)_r$.
We also know that $\alpha_r<\wb_r$. By \eqref{eq:11} below,
if $j<r$ and
$\mu_j<(\lambda-\rho)_r$, then $\alpha_j<\wb_r-(s-1) \le \wb_r$
(since $s \ge 1$). Therefore by
Lemma~\ref{lem:rec e} we have
\begin{equation}\label{eq:e=ye+e}
   E^{\alpha,\wb}_{\lambda-\rho,\mu}(C) 
   = y E^{\alpha,\wb}_{\lambda-\rho-\epsilon_r,\mu}(C\cup\{r\})
       +E_{\lambda-\rho,\mu}^{\alpha,\wb-\epsilon_r}(C\setminus\{r\}).
\end{equation}

For $1\le i\le s-1$, by \eqref{eq:8} below, we have $\alpha_r < \wb_r-(s-1) \le
\wb_r-i =(\wb-i\epsilon_r)_r$
(whence $\wb-i\epsilon_r \in \NN^n$),
and furtheremore, by \eqref{eq:11} below, we
have
$\alpha_j<\wb_r-(s-1)\le (\wb-i\epsilon_r)_r$ for each $j<r$
satisfying
$\mu_j<(\lambda-\rho)_r$. Therefore we can apply Lemma~\ref{lem:rec e}
repeatedly to $E^{\alpha,\wb-i\epsilon_r}_{\lambda-\rho,\mu}(C\setminus\{r\})$,
for $i=1,2,\dots,s-1$, to get
  \begin{align*}
   E_{\lambda-\rho,\mu}^{\alpha,\wb-\epsilon_r}(C\setminus\{r\})
   &= E_{\lambda-\rho,\mu}^{\alpha,\wb-2\epsilon_r}(C\setminus\{r\})
       +x_{\wb_r-1} E_{\lambda-\rho-\epsilon_r,\mu}^{\alpha,\wb-\epsilon_r}(C\cup\{r\}),\\
   E_{\lambda-\rho,\mu}^{\alpha,\wb-2\epsilon_r}(C\setminus\{r\})
   &= E_{\lambda-\rho,\mu}^{\alpha,\wb-3\epsilon_r}(C\setminus\{r\})
     +x_{\wb_r-2} E_{\lambda-\rho-\epsilon_r,\mu}^{\alpha,\wb-2\epsilon_r}(C\cup\{r\}),\\
    \vdots\\
   E_{\lambda-\rho,\mu}^{\alpha,\wb-(s-1)\epsilon_r}(C\setminus\{r\})
   &= E_{\lambda-\rho,\mu}^{\alpha,\wb-s\epsilon_r}(C\setminus\{r\})
     +x_{\wb_r-(s-1)} E_{\lambda-\rho-\epsilon_r,\mu}^{\alpha,\wb-(s-1)\epsilon_r}(C\cup\{r\}).
  \end{align*}
  Since $s\ge1$, combining \eqref{eq:e=ye+e} and the above equations yields
\begin{align}\label{eq:E=yE}
E^{\alpha,\wb}_{\lambda-\rho,\mu}(C) 
   &= y E^{\alpha,\wb}_{\lambda-\rho-\epsilon_r,\mu}(C\cup\{r\}) 
  +E_{\lambda-\rho,\mu}^{\alpha,\wb-s\epsilon_r}(C\setminus\{r\}) \\
  \notag &\qquad + \sum_{i=1}^{s-1}
       x_{\wb_r-i} E_{\lambda-\rho-\epsilon_r,\mu}^{\alpha,\wb-i\epsilon_r}(C\cup\{r\}).
  \end{align}
  The second summand of the right hand side of \eqref{eq:E=yE} vanishes:
  \begin{equation}\label{eq:ECr=0}
E_{\lambda-\rho,\mu}^{\alpha,\wb-s\epsilon_r}(C\setminus\{r\}) = 0,
  \end{equation}
  because of Lemma~\ref{lem:s} below.

In view of $C=C(\rho)$ and $\wb=B(R_0,\beta)$,
we can rewrite \eqref{eq:defEE} as
  \[
  \EE_{\lambda,\mu}^{\alpha,\beta}(R_0) = \owt(R_0) 
  E_{\lambda-\rho,\mu}^{\alpha,\wb}(C).
\]
Thus, multiplying both sides of \eqref{eq:E=yE} by
$\owt(R_0)$ and using \eqref{eq:ECr=0} we obtain
\[
  \EE^{\alpha,\beta}_{\lambda,\mu}(R_0) = y\owt(R_0)
  E^{\alpha,\wb}_{\lambda-\rho-\epsilon_r,\mu}(C\cup\{r\}) + \sum_{i=1}^{s-1}
  x_{\wb_r-i} \owt(R_0) E_{\lambda-\rho-\epsilon_r,\mu}^{\alpha,\wb-i\epsilon_r}(C\cup\{r\}).
  \]
  One can easily check that $y\owt(R_0) = \owt(R_0\cup\{\wb_r\})$ and
  $x_{\wb_r-i} \owt(R_0)=\owt(R_0\cup\{\wb_r-i\})$ for $1\le i\le s-1$. Hence we
  can rewrite the above equation as
\[
  \EE^{\alpha,\beta}_{\lambda,\mu}(R_0) =\sum_{i=0}^{s-1}
  \owt(R_0\cup\{\wb_r-i\})
  E_{\lambda-\rho-\epsilon_r,\mu}^{\alpha,\wb-i\epsilon_r}(C\cup\{r\})
  = \sum_{i=0}^{s-1} \EE^{\alpha,\beta}_{\lambda,\mu}(R_0\cup\{\wb_r-i\}),
\]
where the last equality follows from $\wb-i\epsilon_r =
B(R_0\cup\{\wb_r-i\},\beta)$.
Since $s = \wb_r - a + 1$, this is equivalent to \eqref{eq:rec E} and the
proof is completed.
\end{proof}

The following lemma proves two statements used in the proof of
Proposition~\ref{prop:rec RPP R0}.

\begin{lem}\label{lem:s-1}
  Following the notation in Proposition~\ref{prop:rec RPP R0} 
  and letting
        \[
s = \begin{cases}
  \wb_r-\wb_{r-1}+1, & \mbox{if $(r-1,c)\in\rho$
  and $\wb_{r-1}\ge\alpha_r+1$} ,\\
\wb_r-\alpha_r, &\mbox{otherwise},
\end{cases}
  \]
 we have
 \begin{equation}
   \label{eq:8a}
   \alpha_r < \wb_r,
 \end{equation}
  \begin{equation}
    \label{eq:s>0}
   s \ge 0,
 \end{equation}
 \begin{equation}
   \label{eq:s=0-cond}
 \mbox{if $s = 0$, then }  r \notin C(\rho),
 \end{equation}
  \begin{equation}
    \label{eq:s>1}
\mbox{if $\beta_{r-1} \leq \beta_r$, then }   s \ge 1,
 \end{equation}
  \begin{equation}
    \label{eq:8}
     \alpha_r <\wb_r-(s-1), \mbox{ and}
  \end{equation}
  \begin{equation}
    \label{eq:11}
 \mbox{if $j<r$ and $\mu_j<(\lambda-\rho)_r$, then }   \alpha_j<\wb_r-(s-1).
  \end{equation}
\end{lem}
\begin{proof}
  The first statement \eqref{eq:8a} is easiest to prove:
  If $r\notin C(\rho)$, then $\wb_r=\beta_r\ge\alpha_r+1$
  (since $\alpha < \beta$);
  if $r\in C(\rho)$, then $\wb_r=R_0(r,\lambda_r-\rho_r+1)\ge\alpha_r+1$
  (since $R_0 \in \RPP^{\row(\alpha,\beta)}_{\rho}$).
  In either case we have $\wb_r\ge \alpha_r+1$,
  and thus \eqref{eq:8a} holds.

  We will prove the next two statements \eqref{eq:s>0} and \eqref{eq:s=0-cond}
  by considering the two cases in the definition of $s$.

  \textbf{Case 1:}
  $(r-1,c)\in\rho$ and $\wb_{r-1}\ge\alpha_r+1$. Then
  $s=\wb_r-\wb_{r-1}+1$. Note that $\beta_{r-1} \le \beta_r+1$
  since $\beta \in \QPar_n$.
  Observe that $r-1\in C(\rho)$ and $(r-1,c)$ is the
  leftmost cell in the $(r-1)$st row of $\rho$. Thus $\wb_{r-1}=R_0(r-1,c)$. If
  $r\notin C(\rho)$, we have
  \begin{equation}
    \label{eq:17}
  \wb_{r-1}=R_0(r-1,c) \le \beta_{r-1} \le \beta_r+1 = \wb_r+1
  \end{equation}
  by the definition of $\wb_r$, and thus
  $s=\wb_r-\wb_{r-1}+1\ge0$.
  If $r\in C(\rho)$, then $(r,c+1)$ is the leftmost cell of the $r$th row of
  $\rho$ and therefore
  \begin{equation}
    \label{eq:17b}
  \wb_{r-1}=R_0(r-1,c)  \le R_0(r,c+1) = \wb_r
  \end{equation}
  and thus $s=\wb_r-\wb_{r-1}+1\ge1\ge 0$.
Thus we always have $s\ge 0$.
This means that \eqref{eq:s>0} holds.

If $r \in C(\rho)$, then $s=\wb_r-\wb_{r-1}+1 \ge 1$ by \eqref{eq:17b},
and thus $s \neq 0$. Taking the contrapositive yields \eqref{eq:s=0-cond}.
So we have proved both \eqref{eq:s>0} and \eqref{eq:s=0-cond} in this case.

\textbf{Case 2:}
$(r-1,c)\notin\rho$ or $\wb_{r-1}<\alpha_r+1$.
Then, $s=\wb_r-\alpha_r$. From \eqref{eq:8a},
we have $\wb_r \ge \alpha_r + 1$.
Therefore $s=\wb_r-\alpha_r \ge 1$ (so $s=0$ cannot happen in this case),
which completes the proof of the
two statements \eqref{eq:s>0} and \eqref{eq:s=0-cond}.

The fourth statement \eqref{eq:s>1} is proved just as we proved
\eqref{eq:s>0}, except that $\beta_{r-1} \le \beta_r+1 = \wb_r+1$
is replaced by $\beta_{r-1} \leq \beta_r = \wb_r$ in \eqref{eq:17}.

Now we prove the fifth statement \eqref{eq:8}. By definition of $s$ we have
\begin{equation}
  \label{eq:13}
\wb_r-(s-1) = \begin{cases}
  \wb_{r-1}, & \mbox{if $(r-1,c)\in\rho$
  and $\wb_{r-1}\ge\alpha_r+1$} ,\\
\alpha_r+1, &\mbox{otherwise},
\end{cases}
\end{equation}
which implies $\wb_r-(s-1)\ge\alpha_r+1$. This shows the fifth statement
\eqref{eq:8}.
  
  For the last statement \eqref{eq:11} suppose that $j<r$ and
  $\mu_j<(\lambda-\rho)_r$. We have $\mu_j<(\lambda-\rho)_r=c$ by
  the definition of $\rho$, thus $(j,c)\notin\mu$. In view of
  $j<r$, this leads to $(j,c) \in \lm$ and consequently
  $(j,c)\in \rho$. Hence, $(r-1,c)\in \rho$ (since $j \le r-1 < r$),
  and thus (as in Case 1 above) $\wb_{r-1}=R_0(r-1,c)$.
  Now
\[
\alpha_j < \alpha_j+1 \le R_0(j,c) \le R_0(r-1,c) = \wb_{r-1}\le \wb_r-(s-1),
\]
where the last inequality follows from \eqref{eq:13} because $(r-1,c)\in\rho$.
This shows \eqref{eq:11}.
\end{proof}

Now we prove the identity \eqref{eq:ECr=0} used in the proof of
Proposition~\ref{prop:rec RPP R0}.

\begin{lem}\label{lem:s}
  Following the notation in Proposition~\ref{prop:rec RPP R0} we have
  \begin{equation}\label{eq:e0}
E_{\lambda-\rho,\mu}^{\alpha,\wb-s\epsilon_r}(C\setminus\{r\}) = 0,
  \end{equation}
  where $C=C(\rho)$ and 
      \[
s = \begin{cases}
  \wb_r-\wb_{r-1}+1, & \mbox{if $(r-1,c)\in\rho$
  and $\wb_{r-1}\ge\alpha_r+1$} ,\\
\wb_r-\alpha_r, &\mbox{otherwise} . 
\end{cases}
  \]
\end{lem}
\begin{proof}
  Clearly, $\lambda - \rho \in \NN^n$.
  Also, \eqref{eq:8} shows that $\wb_r - s \ge \alpha_r \ge 0$,
  whence $\wb - s \epsilon_r \in \NN^n$.
  We consider the two cases in the definition of $s$ separately.

  \textbf{Case 1:} $(r-1,c)\in\rho$ and $\wb_{r-1}\ge\alpha_r+1$. Then
  $s=\wb_r-\wb_{r-1}+1$. In order to prove \eqref{eq:e0} it suffices to check
  that $\alpha,\wb-s\epsilon_r, \lambda-\rho,\mu, C\setminus\{r\}$, and $r$
  satisfy the conditions for $\alpha,\beta,\lambda,\mu,C$, and $r$ in
  Lemma~\ref{lem:equal rows}. Since $(r-1,c)\in\rho$ and $\rho$ is the set of
  all cells $(i,j)\prec(r,c)$ we have $(\lambda-\rho)_r =
  (\lambda-\rho)_{r-1}+1$. The fact $(r-1,c)\in\rho$ also implies $r-1\in
  C\setminus\{r\}$, and clearly $r\notin C\setminus\{r\}$. Since
  $s=\wb_r-\wb_{r-1}+1$, we have $(\wb-s\epsilon_r)_r =
  (\wb-s\epsilon_r)_{r-1}-1$. Therefore the conditions in Lemma~\ref{lem:equal
    rows} hold, and \eqref{eq:e0} is proved in this case.

  \textbf{Case 2:} $(r-1,c)\notin\rho$ or $\wb_{r-1}<\alpha_r+1$. Then
  $s=\wb_r-\alpha_r$. Let $d$ be the integer such that $(d,c)\in\lm$ and
  $(d-1,c)\notin\lm$. In other words, $(d,c)$ is the topmost cell in the $c$th
  column of $\lm$, see Figure~\ref{fig:d}. Then $1\le d\le r$. Let
  $\kappa=\wb-s\epsilon_r$ and $\sigma =\lambda-\rho$, so that
\[
  E_{\lambda-\rho,\mu}^{\alpha,\wb-s\epsilon_r}(C\setminus\{r\})
  =  E_{\sigma,\mu}^{\alpha,\kappa}(C\setminus\{r\})
=  \det(e_{\sigma,\mu}^{\alpha,\kappa}(C\setminus\{r\};i,j))_{1\le i,j\le n}.
\]

\begin{figure}
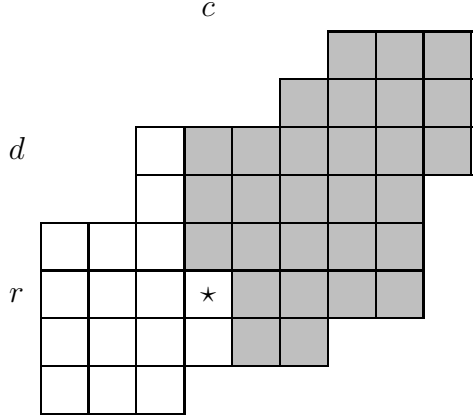

  \centering
\begin{ytableau}
\none & \none & \none &\none& \none[c]  &\none  & \none  & \none & \none & \none\\
\none & \none & \none &\none&\none  &\none  & \none  & *(lightgray) & *(lightgray) & *(lightgray)\\
\none & \none & \none &\none&\none  &\none  & *(lightgray)  & *(lightgray) & *(lightgray) & *(lightgray)\\
\none[d] & \none & \none & & *(lightgray) & *(lightgray) & *(lightgray)  & *(lightgray) & *(lightgray) & *(lightgray)\\
\none & \none & \none & & *(lightgray) & *(lightgray) & *(lightgray)  & *(lightgray) & *(lightgray) \\
\none &  &  & & *(lightgray) & *(lightgray) & *(lightgray)  & *(lightgray) & *(lightgray) \\
\none[r] &  &  & & \star & *(lightgray) & *(lightgray)  & *(lightgray) & *(lightgray) \\
\none &  &  & &  & *(lightgray) & *(lightgray)    \\
\none &  &  & 
\end{ytableau}
\caption{An example of $\lm$ and $\rho$, where the cells in $\rho$ are the gray
  cells. The $(r,c)$ cell is marked with a star and the row indices $d$ and $r$,
  and the column index $c$ are shown.}
  \label{fig:d}
\end{figure}

 We claim that 
\begin{equation}\label{eq:E=dd}
  E_{\sigma,\mu}^{\alpha,\kappa}(C\setminus\{r\})
  =  \det(e_{\sigma,\mu}^{\alpha,\kappa}(C\setminus\{r\};i,j))_{1\le i,j\le d-1}
  \det(e_{\sigma,\mu}^{\alpha,\kappa}(C\setminus\{r\};i,j))_{d\le i,j\le n}.
\end{equation}
To show \eqref{eq:E=dd}, by Lemma~\ref{lem:det=detdet} it suffices to show that
for all $d\le i\le n$ and $1\le j\le d-1$, we have
\begin{equation}
  \label{eq:12}
 e_{\sigma,\mu}^{\alpha,\kappa}(C\setminus\{r\};i,j) = 0.
\end{equation}
By the definitions of $\rho$ and $d$ (see Figure~\ref{fig:d}), if $d\le i\le n$
and $1\le j\le d-1$, we have
\[
\sigma_i=(\lambda-\rho)_i \le c \le \mu_j
\]
(since $(j, c) \in \mu$)
 and therefore
\begin{equation}\label{eq:su}
 \sigma_i-i-\mu_j+j < \sigma_i-\mu_j \le 0. 
\end{equation}
Then \eqref{eq:12} follows from \eqref{eq:su} and the claim \eqref{eq:E=dd} is proved.

By \eqref{eq:E=dd}, to show \eqref{eq:e0} it suffices to show that
\begin{equation}
  \label{eq:15}
\det(e_{\sigma,\mu}^{\alpha,\kappa}(C\setminus\{r\};i,j))_{d\le i,j\le n} = 0.  
\end{equation}
By Lemma~\ref{lem:det=0}, to show \eqref{eq:15} it is enough to show
\eqref{eq:12} for all $d\le i\le r$ and $r\le j\le n$.

By Definition~\ref{defn:E} we have
    \[
      e^{\alpha,\kappa}_{\sigma,\mu}(C\setminus\{r\};i,j)=
      \begin{cases}
      e_{\sigma_i-i-\mu_j+j}[X_{(\alpha_j,\kappa_i-1]} +
      T_{\sigma_i}-T_{\mu_j}],
      & \mbox{if $i\in C\setminus\{r\}$},\\
      e_{\sigma_i-i-\mu_j+j}[X_{(\alpha_j,\kappa_i]} +
      T_{\sigma_i-1}-T_{\mu_j}],
      & \mbox{if $i\not\in C\setminus\{r\}$}.
      \end{cases}
  \]
Suppose $i=r$ and $r\le j\le n$. Then the above equation becomes
  \begin{equation}\label{eq:XTT}
    e^{\alpha,\kappa}_{\sigma,\mu}(C\setminus\{r\};r,j)
    =e_{\sigma_r-r-\mu_j+j}[X_{(\alpha_j,\kappa_r]} + T_{\sigma_r-1}-T_{\mu_j}].
  \end{equation}
  Since
  \[
\sigma_r=(\lambda-\rho)_r>\mu_r\ge\mu_j,
\]
the sum
$X_{(\alpha_j,\kappa_r]} + T_{\sigma_r-1}-T_{\mu_j}$ is a sum of 
$\max\{0,\kappa_r-\alpha_j\} +\sigma_r-1-\mu_j$ variables. Since
\[
\kappa_r = (\wb-s\epsilon_r)_r = \wb_r-s = \alpha_r \le \alpha_j + j-r
\]
(a consequence of $\alpha \in \QPar_n$ and $r \le j$)
and therefore $\kappa_r-\alpha_j \le j-r$
and consequently $\max\{0,\kappa_r-\alpha_j\}
\le \max\{0,j-r\} = j-r$ (since $r\le j$), we have
\[
\max\{0,\kappa_r-\alpha_j\} +\sigma_r-1-\mu_j
\le j-r+\sigma_r-1-\mu_j < \sigma_r-r-\mu_j+j.
\]
Thus, \eqref{eq:ekA=0} shows that
the right hand side of \eqref{eq:XTT} is zero, and
\eqref{eq:12} holds for $i = r$ and $r\le j\le n$.

  It remains to prove \eqref{eq:12} for $d\le i\le r-1$ and $r\le j\le n$.
  Recall the assumption that $(r-1,c)\notin\rho$ or $\wb_{r-1}<\alpha_r+1$. If
  $(r-1,c)\notin\rho$, then $r=d$ and there is no integer $i$ with $d\le i\le
  r-1$. Therefore we may assume $(r-1,c)\in\rho$ and $\wb_{r-1}<\alpha_r+1$.
  Since $d\le i\le r-1$, we have $(i,c)\in \rho$, which is the leftmost cell in
  the $i$th row of $\rho$. Thus $\wb_{i} = R_0(i,c)$ and considering the case
  $i=r-1$ we also have $\wb_{r-1} = R_0(r-1,c)$. Then we obtain
\[
\kappa_i = (\wb-s\epsilon_r)_i = \wb_i = R_0(i,c) \le R_0(r-1,c) = \wb_{r-1}<
\alpha_r+1 \le \alpha_j+j-r+1
\]
(a consequence of $\alpha \in \QPar_n$ and $r \le j$),
which shows
\[
\kappa_i-1-\alpha_j < j-r <j-i.
\]
Combined with $0<j-i$ (which is because $i\le r-1<r\le j$), this yields
\begin{equation}\label{eq:kappa}
\max\{0,\kappa_i-1-\alpha_j\} <j-i.
\end{equation}
Since $(i, c) \in \rho$, we have $i\in C(\rho)=C$. Thus $i\in
C\setminus\{r\}$, and consequently
  \begin{equation}\label{eq:20}
    e^{\alpha,\kappa}_{\sigma,\mu}(C\setminus\{r\};i,j) =
    e_{\sigma_i-i-\mu_j+j}[X_{(\alpha_j,\kappa_i-1]} + T_{\sigma_i}-T_{\mu_j}].
  \end{equation}
  On the other hand, since $d\le i\le r-1$ and $j\ge r$, we have
\begin{align*}
\sigma_i &= (\lambda-\rho)_i = c-1 
\qquad \left(\text{since $(i, c) \in \rho$ and $(i, c-1) \notin \rho$}\right) \\
&\ge \mu_r
\qquad \left(\text{since $(r, c) \in \lm$}\right) \\
&\ge \mu_j.
\end{align*}
Therefore
$X_{(\alpha_j,\kappa_i-1]} + T_{\sigma_i}-T_{\mu_j}$
is a sum of $\max\{0,\kappa_i-1-\alpha_j\}+\sigma_i-\mu_j$ variables.
By \eqref{eq:kappa}, we have
\[
  \max\{0,\kappa_i-1-\alpha_j\}+\sigma_i-\mu_j < j-i +\sigma_i-\mu_j
  = \sigma_i-i-\mu_j+j,
\]
which implies (via \eqref{eq:ekA=0})
that the right hand side of \eqref{eq:20} is zero.
Thus we obtain \eqref{eq:12} and the proof is completed.
\end{proof}

The parallel recurrence relations for $\oR^{\alpha,\beta}_{\lambda,\mu}(R_0)$ and
$\EE^{\alpha,\beta}_{\lambda,\mu}(R_0)$ in Proposition~\ref{prop:rec RPP R0} can
be used to conclude that they are equal.

\begin{prop}\label{prop:main e}
  Let $\alpha,\beta\in\QPar_n$, $\lambda,\mu\in\Par_n$ with $\alpha<\beta$ and
  $\mu<\lambda$. Let $\rho=(\lm)^{(m)}$ for some $0\le m\le|\lm|$ and let
  $R_0\in \RPP^{\row(\alpha,\beta)}_{\rho}$. Then
\[
\EE^{\alpha,\beta}_{\lambda,\mu}(R_0) = \oR^{\alpha,\beta}_{\lambda,\mu}(R_0). 
\]
\end{prop}
\begin{proof}
  We use induction on $N = |\lm|-|\rho|$. For the base case, suppose $N=0$ so
  that $\rho=\lm$. Then clearly $\oR^{\alpha,\beta}_{\lambda,\mu}(R_0) =
  \owt(R_0)$ and, by Lemma~\ref{lem:initial E}, 
  \[
\EE^{\alpha,\beta}_{\lambda,\mu}(R_0) = \owt(R_0)
E^{\alpha,B(R_0,\beta)}_{\mu,\mu}(C(\rho)) = \owt(R_0).
  \]  

  For the inductive step let $0<N\le|\lm|$ and assume the assertion for $N-1$.
  Since $\rho\ne\lm$, we can find $(r,c)\in\lm$ such that $\rho=\{(i,j)\in\lm:
  (i,j)\prec (r,c)\}$. By Proposition~\ref{prop:rec RPP R0} and the induction
  hypothesis, we obtain
\[
    \EE^{\alpha,\beta}_{\lambda,\mu}(R_0)
    = \sum_{k=a}^{\wb_r} \EE^{\alpha,\beta}_{\lambda,\mu}(R_0\cup\{k\})
    =\sum_{k=a}^{\wb_r} \oR^{\alpha,\beta}_{\lambda,\mu}(R_0\cup\{k\})
=\oR^{\alpha,\beta}_{\lambda,\mu}(R_0),
\]
where $a$ and $\wb$ are given as in Proposition~\ref{prop:rec RPP R0}.
Hence the assertion still holds for $N$ and the proof follows by induction. 
\end{proof}

Now we are ready to prove Theorem~\ref{thm:col_flag2}, which can be restated as
follows. 

\begin{thm}\label{thm:flag e2}
  Let $\alpha,\beta\in\NN^n$ and $\mu,\lambda\in\Par_n$. 
If $\alpha_i\le \alpha_{i+1}+1$ and $\beta_i\le \beta_{i+1}+1$ whenever
$\mu_i<\lambda_{i+1}$, then
\begin{equation}
\label{eq:flag e2}
E_{\lambda,\mu}^{\alpha,\beta} = \oR_{\lambda,\mu}^{\alpha,\beta}(\emptyset).
\end{equation}
\end{thm}
\begin{proof}
  We will successively reduce the cases so that we eventually have the
  assumptions $\alpha,\beta\in\QPar_n$, $\alpha<\beta$ and $\mu<\lambda$ in
  Proposition~\ref{prop:main e}. For a diagram $\sigma$, we denote by
  $\delta(\sigma)$ the diagram obtained by translating $\sigma$ down by one row,
  so that $\delta^k(\sigma)=\{(i+k,j): (i,j)\in \sigma\}$ for all $k\ge0$. Note
  that there is a canonical bijection between the RPPs $R$ of shape $\sigma$ and
  the RPPs $R'$ of shape $\delta^k(\sigma)$, and that this bijection satisfies
  $\owt(R') = \owt(R)$.

  If $\mu\not\subseteq\lambda$, both sides of the equation
  \eqref{eq:flag e2} are zero by
  Lemma~\ref{lem:E(C)=0} and the definition of
  $\oR_{\lambda,\mu}^{\alpha,\beta}(\emptyset)$. Hence we may assume
  $\mu\subseteq\lambda$.
  Thus, either $\lambda_k = \mu_k$ for some $1 \le k \le n$, or $\mu < \lambda$.

  Suppose that $\lambda_k=\mu_k$ for some $1\le k\le n$. Then, for $k\le i\le n$
  and $1\le j\le k$, since $\lambda_i-i-\mu_j+j\le \lambda_k-k-\mu_k+k=0$ and
  the equality holds if and only if $i=j=k$, we have
  \[
    e^{\alpha,\beta}_{\lambda,\mu}(i,j) =
    e_{\lambda_i-i-\mu_j+j}[X_{(\alpha_j,\beta_i]}+T_{\lambda_i-1}-T_{\mu_j}] = \chi(i=j=k).
  \]
  By Lemma~\ref{lem:det=det1det} this shows that
  \[
    E_{\lambda,\mu}^{\alpha,\beta} = E_{\lambda^{(1)},\mu^{(1)}}^{\alpha^{(1)},\beta^{(1)}}
    E_{\lambda^{(2)},\mu^{(2)}}^{\alpha^{(2)},\beta^{(2)}},
  \]
  where $\gamma^{(1)}=(\gamma_1,\dots,\gamma_{k-1})$ and
  $\gamma^{(2)}=(\gamma_{k+1},\dots,\gamma_{n})$ for each
  $\gamma\in\{\alpha,\beta,\lambda,\mu\}$. Since the skew shape $\lm$ is the
  disjoint union of $\lambda^{(1)}/\mu^{(1)}$ and $\delta^k(\lambda^{(2)}/\mu^{(2)})$, the
  definition of $\oR_{\lambda,\mu}^{\alpha,\beta}(C)$ immediately gives
  \[
    \oR_{\lambda,\mu}^{\alpha,\beta}(\emptyset) =
    \oR_{\lambda^{(1)},\mu^{(1)}}^{\alpha,\beta}(\emptyset)
    \oR_{\delta^k(\lambda^{(2)}),\delta^k(\mu^{(2)})}^{\alpha,\beta}(\emptyset)=
    \oR_{\lambda^{(1)},\mu^{(1)}}^{\alpha^{(1)},\beta^{(1)}}(\emptyset)
    \oR_{\lambda^{(2)},\mu^{(2)}}^{\alpha^{(2)},\beta^{(2)}}(\emptyset).
  \]
  Hence, by induction, it suffices to consider the case $\mu<\lambda$.

  Suppose that there is an integer $k\in[n-1]$ such that $\mu_k\ge
  \lambda_{k+1}$. Then we have
  \[
    \oR_{\lambda,\mu}^{\alpha,\beta}(\emptyset) =
    \oR_{\lambda^{(1)},\mu^{(1)}}^{\alpha,\beta}(\emptyset)
    \oR_{\delta^k(\lambda^{(2)}),\delta^k(\mu^{(2)})}^{\alpha,\beta}(\emptyset)=
    \oR_{\lambda^{(1)},\mu^{(1)}}^{\alpha^{(1)},\beta^{(1)}}(\emptyset)
    \oR_{\lambda^{(2)},\mu^{(2)}}^{\alpha^{(2)},\beta^{(2)}}(\emptyset).
  \]
  since the skew shape $\lm$ is the disjoint union of $\lambda^{(1)}/\mu^{(1)}$
  and $\delta^k(\lambda^{(2)}/\mu^{(2)})$, where $\gamma^{(1)}=(\gamma_1,\dots,\gamma_k)$
  and $\gamma^{(2)}=(\gamma_{k+1},\dots,\gamma_n)$ for each
  $\gamma\in\{\alpha,\beta,\lambda,\mu\}$. 
  For all $k+1\le i\le n$ and $1\le j\le k$, we have
  $e^{\alpha,\beta}_{\lambda,\mu}(i,j)=0$ because
 \[
    \lambda_i-i-\mu_j+j \le \lambda_{k+1}-(k+1)-\mu_k+k <0.
  \]
  By Lemma~\ref{lem:det=detdet}, this shows that
  \[
    E_{\lambda,\mu}^{\alpha,\beta} = E_{\lambda^{(1)},\mu^{(1)}}^{\alpha^{(1)},\beta^{(1)}}
    E_{\lambda^{(2)},\mu^{(2)}}^{\alpha^{(2)},\beta^{(2)}}.
  \]
  Thus, by induction, we may assume that $\mu_k<\lambda_{k+1}$ for all $k\in[n-1]$. In
  this case by assumption we have $\alpha,\beta\in\QPar_n$.

  Suppose now that $\alpha_k\ge\beta_k$ for some $1\le k\le n$. Then by
  Lemma~\ref{lem:E=0} we have $E_{\lambda,\mu}^{\alpha,\beta}=0$. Again, by
  definition, $\oR_{\lambda,\mu}^{\alpha,\beta}(\emptyset)=0$.

  The remaining case is that $\mu<\lambda$ and $\alpha<\beta$. This is done in
  Proposition~\ref{prop:main e} with $R_0=\emptyset$ and the proof is completed.
\end{proof}

Finally we show that Theorems~\ref{thm:col_flag} and \ref{thm:col_flag2} are equivalent.

\begin{prop}\label{prop:equivalence}
  Let $\alpha,\beta\in\NN^n$ and $\mu,\lambda\in\Par_n$. Then
  \begin{multline*}
\det \left(
    e_{\lambda_i-\mu_j-i+j}(x_{\alpha_j+1},\dots, x_{\beta_i},t_{\mu_j+1},\dots,t_{\lambda_i-1})
  \right)_{1\le i,j\le n}     \\
  = \det \left(
    e_{\lambda_i-\mu_j-i+j}[X_{(\alpha_j,\beta_i]}+T_{\lambda_i-1}-T_{\mu_j}]
  \right)_{1\le i,j\le n}.
  \end{multline*}
\end{prop}
\begin{proof}
  If $\mu\not\subseteq\lambda$, both sides of the equation are zero by the same
  argument as in the proof of Lemma~\ref{lem:E(C)=0}. Hence we may assume that
  $\mu\subseteq\lambda$. Let $A$ and $B$ be the matrices in the left hand side
  and in the right hand side respectively. We investigate the contribution of
  the $(i,j)$-entries $A_{i,j}$ and $B_{i,j}$ in the determinants when
  $A_{i,j}\ne B_{i,j}$.

  Suppose $A_{i,j}\ne B_{i,j}$. Since $\mu_j<\lambda_i$ implies $A_{i,j}=
  B_{i,j}$, we must have $\lambda_i\le \mu_j$. Note that $\lambda_i\le \mu_j\le
  \lambda_j$. If $\lambda_i<\lambda_j$, then $i>j$ and $\lambda_i-\mu_j-i+j<0$,
  which implies $A_{i,j}=B_{i,j}=0$, a contradiction. Thus we must have
  $\lambda_i=\lambda_j$. Since $\lambda_i\le \mu_j\le \lambda_j$, we also have
  $\mu_j=\lambda_j$. We now use an argument in the proof of
  Theorem~\ref{thm:flag e2}. For $j\le r\le n$ and $1\le s\le j$, since
  $\lambda_r-\mu_s-r+s\le \lambda_j-\mu_j-j+j=0$ and the equality holds if and
  only if $r=s=j$, we have
\[
A_{r,s} = B_{r,s} = \chi(r=s=j).
\]
Therefore, by the second claim of Lemma~\ref{lem:det=det1det}, each nonzero term
in the expansion of $\det(A)$ and $\det(B)$ must contain the $(j,j)$ entry,
which is $1$ for both matrices. Thus if $A_{i,j}\ne B_{i,j}$, these entries
$A_{i,j}$ and $B_{i,j}$ do not contribute to the determinants, which implies
$\det(A)=\det(B)$.
\end{proof}

Theorem~\ref{thm:col_flag} now follows from Theorem~\ref{thm:col_flag2} and
Proposition~\ref{prop:equivalence}.

\section{A proof of the Jacobi--Trudi formula for $\wg_{\lambda/\mu}^{\row(\alpha,\beta)}$}
\label{sec:proof-jacobi-trudi2}

In this section we prove the Jacobi--Trudi formula for
$\wg_{\lambda/\mu}^{\row(\alpha,\beta)}$ in Theorem~\ref{thm:row_flag}. The
proof is similar to (but not exactly the same as) that in the previous section.

We use the notation in Definitions~\ref{defn:R0} and \ref{defn:R} from the
previous section. The notation below will also be used throughout this section.

\begin{defn}\label{defn:H}
For $\alpha,\beta,\mu,\lambda\in\NN^n,C\subseteq[n]$, and $1\le i,j\le n$,
define
\begin{align*}
  h_{\lambda,\mu}^{\alpha,\beta}(i,j) &=
  h_{\lambda_i-i-\mu_j+j}[X_{(\alpha_j,\beta_i]}+T_{i-1}-T_{j-1}],\\
  \ovh_{\lambda,\mu}^{\alpha,\beta}(i,j)
  &= \chi(\alpha_j<\beta_i) h_{\lambda,\mu}^{\alpha,\beta}(i,j),\\ 
  h_{\lambda,\mu}^{\alpha,\beta}(C;i,j)
  &= \begin{cases}
      h_{\lambda,\mu}^{\alpha,\beta}(i,j), & \mbox{if $i\not\in C$},\\
      \ovh_{\lambda,\mu}^{\alpha,\beta}(i,j), & \mbox{if $i\in C$},
    \end{cases}\\
  H_{\lambda,\mu}^{\alpha,\beta}(C)
  &= \det \left(h_{\lambda,\mu}^{\alpha,\beta}(C;i,j) \right)_{1\le i,j\le n},\\
  H_{\lambda,\mu}^{\alpha,\beta}
  &= H_{\lambda,\mu}^{\alpha,\beta}(\emptyset),\\
  \ovH_{\lambda,\mu}^{\alpha,\beta}
  &= H_{\lambda,\mu}^{\alpha,\beta}([n]).
\end{align*}
If $\mu\subseteq\lambda$, $\rho=(\lm)^{(m)}$, and
$R_0\in\RPP^{\row(\alpha,\beta)}_\rho$, we define
\begin{equation}
  \label{eq:defHH}
  \HH_{\lambda,\mu}^{\alpha,\beta}(R_0) = \wt(R_0) 
  \ovH_{\lambda-\rho,\mu}^{\alpha,B(R_0,\beta)}.
\end{equation}
\end{defn}

Note that in the definition of $\HH_{\lambda,\mu}^{\alpha,\beta}(R_0)$, we used
$\ovH$ instead of $H$. Using the notation above, Theorem~\ref{thm:row_flag} can
be rewritten as
\[
H_{\lambda,\mu}^{\alpha,\beta} = \R_{\lambda,\mu}^{\alpha,\beta}(\emptyset).
\]
We will show that both sides of the above equation satisfy the same
recurrence relation.

\subsection{Technical lemmas}
\label{sec:technical-lemmas-1}

In this subsection we give a list of lemmas that will be used to prove
Theorem~\ref{thm:row_flag}.

Let $\RPar_n$ denote the set of $\alpha\in\NN^n$ such that $\alpha_i\le
\alpha_{i+1}$ for all $i\in[n-1]$.
Note that if $\alpha\in\RPar_n$ and $1\le i\le j\le n$,
then $\alpha_i\le \alpha_j$.
Clearly, $\RPar_n \subseteq \QPar_n$, where $\QPar_n$ is defined as in
Section~\ref{sec:proof-jacobi-trudi1}.

\begin{lem}\label{lem:pleth_h}
  Let $Z$ be any formal power series with integer coefficients and let $z$ be a
  (single) variable. Then, for any integer $k$,
  \[
h_k[Z]=h_k[Z-z] + z h_{k-1}[Z].
  \]  
\end{lem}
\begin{proof}
  Since $h_k(x)=0$ for $k<0$ and $h_0(x)=1$, the equation holds for $k\le 0$.
  For $k\ge1$, by \eqref{eq:h_k}, we have
  \[
    h_k[Z-z]=\sum_{i=0}^k (-1)^i e_i(z)h_{k-i}[Z]=
    h_{k}[Z]- zh_{k-1}[Z],
  \]  
which is equivalent to the equation in the lemma.
\end{proof}

\begin{lem}\label{lem:h=0}
  Let $i,j,k$ be positive integers such that
  $j\ge i$ and $k>j-i$. Then
  \[
    h_k[T_{i-1}-T_{j-1}]=0.
  \] 
\end{lem}
\begin{proof}
  By the property \eqref{eq:h_k[-A]} of the plethystic substitution, we have
  \[
    h_k[T_{i-1}-T_{j-1}]=h_k[-t_i-t_{i+1}-\dots-t_{j-1}]
    = (-1)^k e_k(t_i,t_{i+1},\dots,t_{j-1}) = 0,
  \]  
  as desired.
\end{proof}

\begin{lem}\label{lem:initial H}
  Let $\alpha,\beta\in\NN^n$, $\mu\in\Par_n$, and $C\subseteq[n]$. If
  $\alpha<\beta$, then
\[
  H^{\alpha,\beta}_{\mu,\mu}(C) = 1.
\] 
\end{lem}
\begin{proof}
  This can be proved by the same argument as in the proof of
  Lemma~\ref{lem:initial E}. Note that we need the condition $\alpha<\beta$ to
  ensure that the $(i,i)$ entry 
\[
h^{\alpha,\beta}_{\mu,\mu}(C;i,i) =
  \chi(\alpha_i<\beta_i) h^{\alpha,\beta}_{\mu,\mu}(i,i)
\]
is $1$ for $i\in C$.
\end{proof}

\begin{lem}\label{lem:H=0}
  Let $\alpha,\beta\in\NN^n$ and $\lambda,\mu\in\Par_n$. If
  $\mu\not\subseteq\lambda$, then for any $C\subseteq[n]$,
\[
  H^{\alpha,\beta}_{\lambda,\mu}(C) = 0.
\] 
\end{lem}
\begin{proof}
This can be proved by the same argument as in the proof of
Lemma~\ref{lem:E(C)=0}. 
\end{proof}

\begin{lem}\label{lem:HC=0 mu<la}
  Let $\alpha,\beta\in\RPar_n$ and $\mu,\lambda\in\Par_n$. Suppose that
  $\alpha_k\ge \beta_k$ and $\mu_k<\lambda_k$ for some $1\le k\le n$. Then for
  any $C\subseteq [n]$,
\[
H^{\alpha,\beta}_{\lambda,\mu}(C)  = 0.
\]
\end{lem}
\begin{proof}
  Since $h^{\alpha,\beta}_{\lambda,\mu}(C;i,j)$ is a multiple of
  $h^{\alpha,\beta}_{\lambda,\mu}(i,j)$, by Lemma~\ref{lem:det=0}, it suffices
  to show $h^{\alpha,\beta}_{\lambda,\mu}(i,j)=0$ assuming $1\le i\le k$ and
  $k\le j\le n$. From $\alpha, \beta \in \RPar_n$, we obtain
  $\alpha_j\ge \alpha_k\ge \beta_k\ge \beta_i$, so
  $X_{(\alpha_j,\beta_i]} = 0$. Because of this, and of
  \[
\lambda_i-i-\mu_j+j\ge \lambda_k-\mu_k-i+j>j-i,
  \]  
  by Lemma~\ref{lem:h=0},
  \[
h^{\alpha,\beta}_{\lambda,\mu}(i,j) = h_{\lambda_i-i-\mu_j+j}[X_{(\alpha_j,\beta_i]}+T_{i-1}-T_{j-1}] = 
h_{\lambda_i-i-\mu_j+j}[T_{i-1}-T_{j-1}]=0,
  \]
  as desired.
\end{proof}

\begin{lem}\label{lem:H(C)=H(C,k)}
  Let $\alpha\in\RPar_n$, $\beta\in\NN^n$, $\mu\in\Par_n$, and $\lambda\in\NN^n$. Suppose
  that there is an integer $1\le k\le n$ such that $\alpha_k<\beta_k$ and
  $\mu_k<\lambda_k$. Then for any subset $C\subseteq[n]$ with $k\in C$, we have
\[
H^{\alpha,\beta}_{\lambda,\mu}(C)  = H^{\alpha,\beta}_{\lambda,\mu}(C\setminus\{k\}).
\]
\end{lem}
\begin{proof}
  It is sufficient to show that for all $1\le j\le n$,
  \begin{equation}
    \label{eq:4}
h^{\alpha,\beta}_{\lambda,\mu}(C;k,j) = h^{\alpha,\beta}_{\lambda,\mu}(C\setminus \{k\};k,j).    
  \end{equation}
  For a contradiction suppose that \eqref{eq:4} does not hold for some
  $1\le j\le n$. Since
  \begin{align*}
    h^{\alpha,\beta}_{\lambda,\mu}(C;k,j) &=
    \chi(\alpha_j<\beta_k) h^{\alpha,\beta}_{\lambda,\mu}(k,j),\\
    h^{\alpha,\beta}_{\lambda,\mu}(C\setminus\{k\};k,j) &=
     h^{\alpha,\beta}_{\lambda,\mu}(k,j),
  \end{align*}
  we must have $\chi(\alpha_j<\beta_k)=0$, or equivalently, $\beta_k\le
  \alpha_j$. Then by assumption we have $\alpha_{k}<\beta_k\le \alpha_j$, which
  implies $k< j$ (since $\alpha \in \RPar_n$). Thus
  \[
    \lambda_k-k-\mu_j+j\ge \lambda_k-k-\mu_k+j > j-k.
  \]
  By Lemma~\ref{lem:h=0},
  \[
    h^{\alpha,\beta}_{\lambda,\mu}(k,j) =
    h_{\lambda_k-k-\mu_j+j}[X_{(\alpha_j,\beta_k]}+T_{k-1}-T_{j-1}] =
    h_{\lambda_k-k-\mu_j+j}[T_{k-1}-T_{j-1}]=0.
  \]
  But this implies that both sides of \eqref{eq:4} are zero, a contradiction.
  Therefore \eqref{eq:4} is true for all $1\le j\le n$, which completes the
  proof.
\end{proof}

\begin{lem}\label{lem:H(1..r)=H}
  Let $\alpha\in\RPar_n$, $\beta\in\NN^n$, and $\mu,\lambda\in\Par_n$ with $\alpha<\beta$
  and $\mu<\lambda$. Then, for any $C\subseteq[n]$,
  \[
H_{\lambda,\mu}^{\alpha,\beta}(C) = H_{\lambda,\mu}^{\alpha,\beta}.
  \]
\end{lem}
\begin{proof}
  We claim that, if $k\in C$, then
  \[
H_{\lambda,\mu}^{\alpha,\beta}(C) = H_{\lambda,\mu}^{\alpha,\beta}(C\setminus\{k\}).
  \]
  By assumption we have $\alpha_k<\beta_k$ and $\mu_k<\lambda_k$. Thus the claim
  follows from Lemma~\ref{lem:H(C)=H(C,k)}. Applying the claim iteratively
  yields the desired result.
\end{proof}

\begin{lem}\label{lem:h=h+xh}
  Let $\alpha\in\RPar_n$, $\beta\in\NN^n$, $\mu\in\Par_n$, and
  $\lambda\in\NN^n$. Suppose that $k$ is an integer satisfying $\alpha_k<
  \beta_k$ and $\mu_k<\lambda_k$. Then
\[
  \ovH_{\lambda,\mu}^{\alpha,\beta} =
  \begin{cases}
  \ovH_{\lambda,\mu}^{\alpha,\beta-\epsilon_k}
  +x_{\beta_k} \ovH_{\lambda-\epsilon_k,\mu}^{\alpha,\beta},
  &\mbox{if $\alpha_k+1<\beta_k$},\\
  H_{\lambda,\mu}^{\alpha,\beta-\epsilon_k}([n]\setminus\{k\})+
  x_{\beta_k} \ovH_{\lambda-\epsilon_k,\mu}^{\alpha,\beta},
  &\mbox{if $\alpha_k+1=\beta_k$}.
  \end{cases}
\]
\end{lem}
\begin{proof}
We first claim that, for all $1\le j\le n$,
  \begin{equation}
    \label{eq:3}
  h_{\lambda,\mu}^{\alpha,\beta}(k,j)
  =h_{\lambda,\mu}^{\alpha,\beta-\epsilon_k}(k,j)
  +x_{\beta_k} \ovh_{\lambda-\epsilon_k,\mu}^{\alpha,\beta}(k,j).
  \end{equation}
  The claim is restated as
  \begin{multline*}
    h_{\lambda_k-k-\mu_j+j}[X_{(\alpha_j,\beta_k]}+T_{k-1}-T_{j-1}]
    =h_{\lambda_k-k-\mu_j+j}[X_{(\alpha_j,\beta_k-1]}+T_{k-1}-T_{j-1}]\\
    +x_{\beta_k} \chi(\alpha_j<\beta_k)
    h_{\lambda_k-k-\mu_j+j-1}[X_{(\alpha_j,\beta_k]}+T_{k-1}-T_{j-1}].
  \end{multline*}
  If $\alpha_j\ge \beta_k$, both sides of the above equation are equal to
  $h_{\lambda_k-k-\mu_j+j}[T_{k-1}-T_{j-1}]$. If $\alpha_j<\beta_k$, the
  equation follows from Lemma~\ref{lem:pleth_h} with
  $Z=X_{(\alpha_j,\beta_k]}+T_{k-1}-T_{j-1}$ and $z=x_{\beta_k}$. This
  establishes \eqref{eq:3}.

By Lemma~\ref{lem:H(C)=H(C,k)} and \eqref{eq:3},
\begin{equation}\label{eq:ovH=H+ovH}
  \ovH_{\lambda,\mu}^{\alpha,\beta} = H_{\lambda,\mu}^{\alpha,\beta}([n])
  = H_{\lambda,\mu}^{\alpha,\beta}([n]\setminus\{k\})
 = H_{\lambda,\mu}^{\alpha,\beta-\epsilon_k}([n]\setminus\{k\})
  +x_{\beta_k} \ovH_{\lambda-\epsilon_k,\mu}^{\alpha,\beta},
\end{equation}
where the last equality follows from the linearity of the determinant in its
$k$th row. This shows the lemma for the case $\alpha_k+1=\beta_k$. If
$\alpha_k+1<\beta_k$, Lemma~\ref{lem:H(C)=H(C,k)} gives
\[
  H_{\lambda,\mu}^{\alpha,\beta-\epsilon_k}([n]\setminus\{k\})
  = H_{\lambda,\mu}^{\alpha,\beta-\epsilon_k}([n])
  =\ovH_{\lambda,\mu}^{\alpha,\beta-\epsilon_k},
\]
which together with \eqref{eq:ovH=H+ovH} finishes the proof.
\end{proof}

\begin{lem}\label{lem:H=H+tH}
  Let $\alpha,\beta,\mu,\lambda\in\NN^n$, and $2\le k\le n$
  with $\beta_k=\beta_{k-1}$ and $\lambda_k=\lambda_{k-1}+1$. Then 
\[
  \ovH_{\lambda,\mu}^{\alpha,\beta} = t_{k-1}
  \ovH_{\lambda-\epsilon_k,\mu}^{\alpha,\beta}.
\]
\end{lem}
\begin{proof}
  We claim that, for all $1\le j\le n$,
  \begin{equation}\label{eq:h=h+th}
  \ovh_{\lambda,\mu}^{\alpha,\beta}(k,j)
  =\ovh_{\lambda,\mu}^{\alpha,\beta}(k-1,j)
  +t_{k-1} \ovh_{\lambda-\epsilon_k,\mu}^{\alpha,\beta}(k,j).
\end{equation}
To prove the claim, since $\chi(\alpha_j<\beta_k)=\chi(\alpha_j<\beta_{k-1})$,
it suffices to show that
  \begin{multline*}
    h_{\lambda_k-k-\mu_j+j}[X_{(\alpha_j,\beta_k]}+T_{k-1}-T_{j-1}]
    =h_{\lambda_{k-1}-(k-1)-\mu_j+j}[X_{(\alpha_j,\beta_{k-1}]}+T_{k-2}-T_{j-1}]\\
    +t_{k-1} h_{\lambda_k-k-\mu_j+j-1}[X_{(\alpha_j,\beta_k]}+T_{k-1}-T_{j-1}],
  \end{multline*}
  which is, by assumption, the same as
  \begin{multline*}
    h_{\lambda_k-k-\mu_j+j}[X_{(\alpha_j,\beta_k]}+T_{k-1}-T_{j-1}]
    =h_{\lambda_{k}-k-\mu_j+j}[X_{(\alpha_j,\beta_{k}]}+T_{k-2}-T_{j-1}]\\
    +t_{k-1} h_{\lambda_k-k-\mu_j+j-1}[X_{(\alpha_j,\beta_k]}+T_{k-1}-T_{j-1}].
  \end{multline*}
  This follows from Lemma~\ref{lem:pleth_h} with
  $Z=X_{(\alpha_j,\beta_k]}+T_{k-1}-T_{j-1}$ and $z=t_{k-1}$.

  Using \eqref{eq:h=h+th} and subtracting the $(k-1)$st row from the $k$th row
  of the matrix for the determinant $\ovH_{\lambda,\mu}^{\alpha,\beta}$ we
  obtain the lemma.
\end{proof}

\subsection{Proof of Theorem~\ref{thm:row_flag}}

We first show that $\R^{\alpha,\beta}_{\lambda,\mu}(R_0)$ and
$\HH^{\alpha,\beta}_{\lambda,\mu}(R_0)$ satisfy the same recurrence relation
under certain conditions.

\begin{prop}\label{prop:H rec RPP R0}
  Let $\alpha,\beta\in\RPar_n$, $\lambda,\mu\in\Par_n$ with $\alpha<\beta$ and
  $\mu<\lambda$. Fix $(r,c)\in\lm$ and $R_0\in
  \RPP^{\row(\alpha,\beta)}_{\rho}$, where $\rho$ is the set of cells
  $(i,j)\in\lm$ with $(i,j)\prec(r,c)$. Let
  $\wb=(\wb_1,\dots,\wb_n)=B(R_0,\beta)$. Then
  \begin{align}
    \label{eq:rec M2}
    \R^{\alpha,\beta}_{\lambda,\mu}(R_0) 
    &= \sum_{k=a}^{\wb_r} \R^{\alpha,\beta}_{\lambda,\mu}(R_0\cup\{k\}),\\
    \label{eq:rec H}
    \HH^{\alpha,\beta}_{\lambda,\mu}(R_0)
    &= \sum_{k=a}^{\wb_r} \HH^{\alpha,\beta}_{\lambda,\mu}(R_0\cup\{k\}),
  \end{align}
  where $R_0\cup\{k\}$ is the RPP obtained from $R_0$ by adding the cell $(r,c)$
  with entry $k$, and
  \[
a = \begin{cases}
  \wb_{r-1}, &\mbox{if $(r-1,c)\in\rho$ and $\wb_{r-1}\ge \alpha_r+1$},\\
 \alpha_r+1, &\mbox{otherwise}.
\end{cases}
\]
\end{prop}
\begin{proof}
  The first identity \eqref{eq:rec M2} is immediate from the definition of
  $\R^{\alpha,\beta}_{\lambda,\mu}(R_0)$. For the second identity \eqref{eq:rec
    H}, let $C=C(\rho)=\{1\le i\le n: \rho_i>0\}$ and
    \[
s = \begin{cases}
  \wb_r-\wb_{r-1}+1, & \mbox{if $(r-1,c)\in\rho$
  and $\wb_{r-1}\ge\alpha_r+1$} ,\\
\wb_r-\alpha_r, &\mbox{otherwise} . 
\end{cases}
  \]
  Note that $\beta \in \RPar_n \subseteq \QPar_n$ and
  $\beta_{r-1} \leq \beta_r$; thus, \eqref{eq:s>1} yields
  $s \ge 1$.
  Furthermore, \eqref{eq:8} yields
  $\wb_r - s \ge \alpha_r \ge 0$, thus
  $\wb - s \epsilon_r \in \NN^n$. Finally,
  $\lambda - \rho \in \NN^n$.
  We consider the two cases in the definition of $s$.

  \textbf{Case 1:} $(r-1,c)\in\rho$ and $\wb_{r-1}\ge\alpha_r+1$. Then
  $s=\wb_r-\wb_{r-1}+1$. Therefore, for $0\le i\le s-2$, we have
\[
(\wb-i\epsilon_r)_r =\wb_r-i > \wb_r -(s-1) = \wb_{r-1}\ge \alpha_r +1.
\]
Note that since $(r,c)\in (\lm)-\rho$, we have $(\lambda-\rho)_r>\mu_r$.
Thus we can apply Lemma~\ref{lem:h=h+xh} repeatedly to
$\ovH^{\alpha,\wb-i\epsilon_r}_{\lambda-\rho,\mu}$, for $0\le i\le s-2$:
\begin{align*}
  \ovH^{\alpha,\wb}_{\lambda-\rho,\mu} &=
  \ovH^{\alpha,\wb-\epsilon_r}_{\lambda-\rho,\mu}
  + x_{\wb_r}\ovH^{\alpha,\wb}_{\lambda-\rho-\epsilon_r,\mu},\\
  \ovH^{\alpha,\wb-\epsilon_r}_{\lambda-\rho,\mu} &=
  \ovH^{\alpha,\wb-2\epsilon_r}_{\lambda-\rho,\mu}
  + x_{\wb_r-1}\ovH^{\alpha,\wb-\epsilon_r}_{\lambda-\rho-\epsilon_r,\mu},\\
  \vdots \\
  \ovH^{\alpha,\wb-(s-2)\epsilon_r}_{\lambda-\rho,\mu} &=
  \ovH^{\alpha,\wb-(s-1)\epsilon_r}_{\lambda-\rho,\mu}
  + x_{\wb_r-(s-2)}\ovH^{\alpha,\wb-(s-2)\epsilon_r}_{\lambda-\rho-\epsilon_r,\mu}.
\end{align*}
Combining the above equations gives
\begin{equation}\label{eq:5.11}
  \ovH^{\alpha,\wb}_{\lambda-\rho,\mu} =
  \ovH^{\alpha,\wb-(s-1)\epsilon_r}_{\lambda-\rho,\mu}
  +\sum_{i=0}^{s-2} x_{\wb_r-i}
  \ovH^{\alpha,\wb-i\epsilon_r}_{\lambda-\rho-\epsilon_r,\mu}.
\end{equation}
Note that if $s=1$, then \eqref{eq:5.11} is trivial. Since $(r-1,c)\in\rho$, by
the construction of $\rho$ we have $(\lambda-\rho)_r = (\lambda-\rho)_{r-1}+1$.
Since
\[
(\wb-(s-1)\epsilon_r)_r = \wb_r-(s-1)= \wb_{r-1} = (\wb-(s-1)\epsilon_r)_{r-1},
\]
Lemma~\ref{lem:H=H+tH} gives
\begin{equation}
  \label{eq:5}
  \ovH^{\alpha,\wb-(s-1)\epsilon_r}_{\lambda-\rho,\mu}  = t_{r-1}
  \ovH^{\alpha,\wb-(s-1)\epsilon_r}_{\lambda-\rho-\epsilon_r,\mu}.
\end{equation}
By \eqref{eq:5.11} and \eqref{eq:5} we obtain 
\begin{equation}
  \label{eq:6}
  \ovH^{\alpha,\wb}_{\lambda-\rho,\mu} = t_{r-1}
  \ovH^{\alpha,\wb-(s-1)\epsilon_r}_{\lambda-\rho-\epsilon_r,\mu}
  +\sum_{i=0}^{s-2} x_{\wb_r-i}
  \ovH^{\alpha,\wb-i\epsilon_r}_{\lambda-\rho-\epsilon_r,\mu}.
\end{equation}

Since $\wb=B(R_0,\beta)$,
we can rewrite \eqref{eq:defHH} as
\[
  \HH_{\lambda,\mu}^{\alpha,\beta}(R_0)
  = \wt(R_0) \ovH_{\lambda-\rho,\mu}^{\alpha,\wb}.
\]
Since 
\begin{align*}
t_{r-1}\wt(R_0) &= \wt(R_0\cup\{\wb_{r-1}\}) = \wt(R_0\cup\{\wb_{r}-(s-1)\}),\\
  x_{\wb_r-i} \wt(R_0)&=\wt(R_0\cup\{\wb_r-i\})
                        \qquad \mbox{for $0 \le i \le s-2$,}
\end{align*}
multiplying both sides of
\eqref{eq:6} by $\wt(R_0)$ gives
\[
  \HH^{\alpha,\beta}_{\lambda,\mu}(R_0)
    = \sum_{i=0}^{s-1} \wt(R_0\cup\{\wb_r-i\}) \ovH_{\lambda-\rho-\epsilon_r,\mu}^{\alpha,\wb-i\epsilon_r}
    = \sum_{i=0}^{s-1} \HH^{\alpha,\beta}_{\lambda,\mu}(R_0\cup\{\wb_r-i\}),
\]  
where the last equality follows from $\wb-i\epsilon_r =
B(R_0\cup\{\wb_r-i\},\beta)$. The above equation is the same as \eqref{eq:rec
  H}.

\textbf{Case 2:} $(r-1,c)\not\in\rho$ or $\wb_{r-1}<\alpha_r+1$. Then
$s=\wb_r-\alpha_r$. As in Case 1 we have $(\lambda-\rho)_r>\mu_r$, and, for
$0\le i\le s-1$,
  \[
(\wb-i\epsilon_r)_r =\wb_r-i \ge \wb_r -(s-1) =  \alpha_r +1,
\]
where the equality holds if and only if $i=s-1$. Thus we can apply
Lemma~\ref{lem:h=h+xh} repeatedly to
$\ovH^{\alpha,\wb-i\epsilon_r}_{\lambda-\rho,\mu}$, for $0\le i\le s-1$:
\begin{align*}
  \ovH^{\alpha,\wb}_{\lambda-\rho,\mu} &=
  \ovH^{\alpha,\wb-\epsilon_r}_{\lambda-\rho,\mu}
  + x_{\wb_r}\ovH^{\alpha,\wb}_{\lambda-\rho-\epsilon_r,\mu},\\
  \ovH^{\alpha,\wb-\epsilon_r}_{\lambda-\rho,\mu} &=
  \ovH^{\alpha,\wb-2\epsilon_r}_{\lambda-\rho,\mu}
  + x_{\wb_r-1}\ovH^{\alpha,\wb-\epsilon_r}_{\lambda-\rho-\epsilon_r,\mu},\\
  \vdots \\
  \ovH^{\alpha,\wb-(s-2)\epsilon_r}_{\lambda-\rho,\mu} &=
  \ovH^{\alpha,\wb-(s-1)\epsilon_r}_{\lambda-\rho,\mu}
  + x_{\wb_r-(s-2)}\ovH^{\alpha,\wb-(s-2)\epsilon_r}_{\lambda-\rho-\epsilon_r,\mu},\\
  \ovH^{\alpha,\wb-(s-1)\epsilon_r}_{\lambda-\rho,\mu} &=
  H^{\alpha,\wb-s\epsilon_r}_{\lambda-\rho,\mu}([n]\setminus\{r\})
  + x_{\wb_r-(s-1)}\ovH^{\alpha,\wb-(s-1)\epsilon_r}_{\lambda-\rho-\epsilon_r,\mu}.
\end{align*}
Lemma~\ref{lem:ss} below shows
$H^{\alpha,\wb-s\epsilon_r}_{\lambda-\rho,\mu}([n]\setminus\{r\})= 0$. Thus,
combining the above equations, we obtain
\[
  \ovH^{\alpha,\wb}_{\lambda-\rho,\mu} =
  \sum_{i=0}^{s-1} x_{\wb_r-i}
  \ovH^{\alpha,\wb-i\epsilon_r}_{\lambda-\rho-\epsilon_r,\mu}.
\]

Observe that, for $0\le i\le s-1$, we have $x_{\wb_r-i}
\wt(R_0)=\wt(R_0\cup\{\wb_r-i\})$ (because we have $(r-1,c) \notin \rho$ or
$\wb_{r-1} < \alpha_r + 1 = \wb_r - (s-1) \le \wb_r - i$) and $\wb-i\epsilon_r =
B(R_0\cup\{\wb_r-i\},\beta)$. Thus, similarly to the argument in Case 1,
multiplying both sides of the above equation by $\wt(R_0)$ gives
\[
  \HH^{\alpha,\beta}_{\lambda,\mu}(R_0)
    = \sum_{i=0}^{s-1} \HH^{\alpha,\beta}_{\lambda,\mu}(R_0\cup\{\wb_r-i\}),
\]  
which is the same as \eqref{eq:rec H}.
This completes the proof.
\end{proof}

We now prove a statement used in the proof of Proposition~\ref{prop:H rec RPP
  R0}.

\begin{lem}\label{lem:ss}
  Following the notation in Proposition~\ref{prop:H rec RPP R0}, suppose
  $(r-1,c)\notin\rho$ or $\wb_{r-1}<\alpha_r+1$, and let $s=\wb_r-\alpha_r$.
  Then we have
  \begin{equation}\label{eq:h0}
H^{\alpha,\wb-s\epsilon_r}_{\lambda-\rho,\mu}([n]\setminus\{r\})= 0.
  \end{equation}
\end{lem}
\begin{proof}
  We proceed similarly as in the proof of Lemma~\ref{lem:s}. Let $d$ be the
  integer such that $(d,c)\in\lm$ and $(d-1,c)\notin\lm$, see
  Figure~\ref{fig:d}. Then $1\le d\le r$. Let $\kappa=\wb-s\epsilon_r$ and
  $\sigma =\lambda-\rho$, so that
\[
H^{\alpha,\wb-s\epsilon_r}_{\lambda-\rho,\mu}([n]\setminus\{r\})
  = H_{\sigma,\mu}^{\alpha,\kappa}([n]\setminus\{r\})
=  \det(h_{\sigma,\mu}^{\alpha,\kappa}([n]\setminus\{r\};i,j))_{1\le i,j\le n}.
\]
By the same argument as in the proof of \eqref{eq:E=dd} we have
\[
  H_{\sigma,\mu}^{\alpha,\kappa}([n]\setminus\{r\})
  =  \det(h_{\sigma,\mu}^{\alpha,\kappa}([n]\setminus\{r\};i,j))_{1\le i,j\le d-1}
  \det(h_{\sigma,\mu}^{\alpha,\kappa}([n]\setminus\{r\};i,j))_{d\le i,j\le n}.
\]
Therefore, it suffices to show that
\[
\det(h_{\sigma,\mu}^{\alpha,\kappa}([n]\setminus\{r\};i,j))_{d\le i,j\le n} = 0.  
\]
By Lemma~\ref{lem:det=0}, in order to show the above equation it is enough to
show that, for all $d\le i\le r$ and $r\le j\le n$,
\begin{equation}
  \label{eq:1}
 h_{\sigma,\mu}^{\alpha,\kappa}([n]\setminus\{r\};i,j) = 0. 
\end{equation}
By Definition~\ref{defn:H} we have
    \[
      h^{\alpha,\kappa}_{\sigma,\mu}([n]\setminus\{r\};i,j)=
      \begin{cases}
      h_{\sigma_i-i-\mu_j+j}[X_{(\alpha_j,\kappa_i]} + T_{i-1}-T_{j-1}],
      & \mbox{if $i= r$},\\
      \chi(\alpha_j<\kappa_i) h_{\sigma_i-i-\mu_j+j}[X_{(\alpha_j,\kappa_i]} + T_{i-1}-T_{j-1}],
      & \mbox{if $i\ne r$}.
      \end{cases}
  \]

  Suppose $i=r$ and $r\le j\le n$. Then the above equation becomes
\[
      h^{\alpha,\kappa}_{\sigma,\mu}([n]\setminus\{r\};r,j)=
      h_{\sigma_r-r-\mu_j+j}[X_{(\alpha_j,\kappa_r]} + T_{r-1}-T_{j-1}].
\]
Since $\kappa_r = (\wb-s\epsilon_r)_r = \wb_r-s = \alpha_r \le \alpha_j$,
this equation can be further simplified to
\begin{equation}
  \label{eq:2}
      h^{\alpha,\kappa}_{\sigma,\mu}([n]\setminus\{r\};r,j)=
      h_{\sigma_r-r-\mu_j+j}[T_{r-1}-T_{j-1}].
\end{equation}
Since $(r,c)\in\rho$, we have
\[
\sigma_r=(\lambda-\rho)_r>\mu_r\ge\mu_j,
\]
which implies
\[
\sigma_r-r-\mu_j+j > j-r.
\]
Thus, Lemma~\ref{lem:h=0} gives
\begin{equation}
  \label{eq:9}
      h_{\sigma_r-r-\mu_j+j}[T_{r-1}-T_{j-1}]=0.
\end{equation}
By \eqref{eq:2} and \eqref{eq:9} we obtain \eqref{eq:1} for $i=r$ and $r\le j\le n$.

It remains to prove \eqref{eq:1} for $d\le i\le r-1$ and $r\le j\le n$. 
In this case,
\[
      h^{\alpha,\kappa}_{\sigma,\mu}([n]\setminus\{r\};i,j)=
      \chi(\alpha_j<\kappa_i) h_{\sigma_i-i-\mu_j+j}[X_{(\alpha_j,\kappa_i]} + T_{i-1}-T_{j-1}].
\]
Thus it suffices to show that $\kappa_i\le\alpha_j$. Since $d\le r-1<r$, we have
$(r-1,c)\in \rho$. Then by assumption in this lemma, $\wb_{r-1}<\alpha_r+1$.
Since $d\le i\le r-1$, we have $(i,c)\in\rho$;
in view of $(i,c-1)\notin\rho$, this leads to $\wb_i = R_0(i,c)$, and
similarly $\wb_{r-1} = R_0(r-1,c)$. Therefore
\[
\kappa_i = (\wb-s\epsilon_r)_i = \wb_i = R_0(i,c) \le R_0(r-1,c) = \wb_{r-1}<
\alpha_r+1 \le \alpha_j+1
\]
(since $\alpha \in \RPar_n$),
which shows $\kappa_i\le\alpha_j$, as desired.
\end{proof}

The fact that $\R^{\alpha,\beta}_{\lambda,\mu}(R_0)$ and
$\HH^{\alpha,\beta}_{\lambda,\mu}(R_0)$ satisfy the same recurrence relation can
be used to show that they are equal.

\begin{prop}\label{prop:main h}
  Let $\alpha,\beta\in\RPar_n$, $\lambda,\mu\in\Par_n$ with $\alpha<\beta$ and
  $\mu<\lambda$. Let $\rho=(\lm)^{(m)}$ for some $0\le m\le|\lm|$ and let
  $R_0\in \RPP^{\row(\alpha,\beta)}_{\rho}$. Then
\[
\HH^{\alpha,\beta}_{\lambda,\mu}(R_0) = \R^{\alpha,\beta}_{\lambda,\mu}(R_0). 
\]
\end{prop}
\begin{proof}
  This can be proved by the same argument as in the proof of
  Proposition~\ref{prop:main e} where we use Lemma~\ref{lem:initial H} and
  Proposition~\ref{prop:H rec RPP R0} in place of Lemma~\ref{lem:initial E} and
  Proposition~\ref{prop:rec RPP R0}, respectively.
\end{proof}

Now we are ready to prove Theorem~\ref{thm:row_flag}, which can be restated as
follows. 

\begin{thm}\label{thm:flag h2}
  Let $\alpha,\beta\in\NN^n$ and $\mu,\lambda\in\Par_n$. 
If
  $\alpha_i\le \alpha_{i+1}$ and $\beta_i\le \beta_{i+1}$ whenever
  $\mu_i<\lambda_{i+1}$, then
\[
H_{\lambda,\mu}^{\alpha,\beta} = \R_{\lambda,\mu}^{\alpha,\beta}(\emptyset).
\]
\end{thm}
\begin{proof}
  As we did in the proof of Theorem~\ref{thm:flag e2} we will successively
  reduce the cases so that we eventually have the assumptions
  $\alpha,\beta\in\RPar_n$, $\alpha<\beta$ and $\mu<\lambda$ in
  Proposition~\ref{prop:main h}. For a diagram $\sigma$, we denote by
  $\delta(\sigma)$ the diagram obtained by translating $\sigma$ down by one row,
  so that $\delta^k(\sigma)=\{(i+k,j): (i,j)\in \sigma\}$ for all $k\ge0$. Let
  $\phi$ be the shifting operator on $\mathbb{Q}[[x_1,x_2,\dots,t_1,t_2,\dots]]$
  replacing each variable $t_i$ by $t_{i+1}$. Then $\phi^k$ is an algebra
  homomorphism and it sends $T_{i-1} - T_{j-1}$ to $T_{i+k-1} - T_{j+k-1}$ for
  all positive integers $i, j, k$. Note that there is a canonical bijection
  between the RPPs $R$ of shape $\sigma$ and the RPPs $R'$ of shape
  $\delta^k(\sigma)$, and that this bijection satisfies $ \wt(R') =
  \phi^k(\wt(R))$.

  If $\mu\not\subseteq\lambda$, both sides are zero by Lemma~\ref{lem:H=0} and
  the definition of $\R_{\lambda,\mu}^{\alpha,\beta}(\emptyset)$. Hence we may
  assume $\mu\subseteq\lambda$. Thus, either $\lambda_k = \mu_k$ for some $1 \le
  k \le n$, or $\mu < \lambda$.

  Suppose that $\lambda_k=\mu_k$ for some $1\le k\le n$. Then for $k\le i\le n$
  and $1\le j\le k$, we have $\lambda_i-i-\mu_j+j\le \lambda_k-k-\mu_k+k=0$,
  where the equality holds if and only if $i=j=k$. Thus
  \[
    h^{\alpha,\beta}_{\lambda,\mu}(i,j) =
    h_{\lambda_i-i-\mu_j+j}[X_{(\alpha_j,\beta_i]}+T_{i-1}-T_{j-1}] = \chi(i=j=k).
  \]
  By Lemma~\ref{lem:det=det1det}, this implies
  \[
    H_{\lambda,\mu}^{\alpha,\beta} = H_{\lambda^{(1)},\mu^{(1)}}^{\alpha^{(1)},\beta^{(1)}}
    \phi^k\left(H_{\lambda^{(2)},\mu^{(2)}}^{\alpha^{(2)},\beta^{(2)}}\right),
  \]
  where $\gamma^{(1)}=(\gamma_1,\dots,\gamma_{k-1})$ and
  $\gamma^{(2)}=(\gamma_{k+1},\dots,\gamma_{n})$ for each
  $\gamma\in\{\alpha,\beta,\lambda,\mu\}$. The definition of
  $\R_{\lambda,\mu}^{\alpha,\beta}(\emptyset)$ immediately gives
  \[
    \R_{\lambda,\mu}^{\alpha,\beta}(\emptyset) =
    \R_{\lambda^{(1)},\mu^{(1)}}^{\alpha,\beta}(\emptyset)
    \R_{\delta^k(\lambda^{(2)}),\delta^k(\mu^{(2)})}^{\alpha,\beta}(\emptyset)
    =\R_{\lambda^{(1)},\mu^{(1)}}^{\alpha^{(1)},\beta^{(1)}}(\emptyset)
    \phi^k\left(\R_{\lambda^{(2)},\mu^{(2)}}^{\alpha^{(2)},\beta^{(2)}}(\emptyset)\right)
  \]
  because $\lm$ is the disjoint union of $\lambda^{(1)}/\mu^{(1)}$ and
  $\delta^k(\lambda^{(2)}/\mu^{(2)})$. Hence, by induction, it suffices to
  consider the case $\mu<\lambda$.

  Suppose that there is an integer $k\in[n-1]$ such that $\mu_k\ge
  \lambda_{k+1}$. Then we have
  \[
    \R_{\lambda,\mu}^{\alpha,\beta}(\emptyset) =
    \R_{\lambda^{(1)},\mu^{(1)}}^{\alpha,\beta}(\emptyset)
    \R_{\delta^k(\lambda^{(2)}),\delta^k(\mu^{(2)})}^{\alpha,\beta}(\emptyset)=
    \R_{\lambda^{(1)},\mu^{(1)}}^{\alpha^{(1)},\beta^{(1)}}(\emptyset)
    \phi^k\left(\R_{\lambda^{(2)},\mu^{(2)}}^{\alpha^{(2)},\beta^{(2)}}(\emptyset)\right),
  \]
  where $\gamma^{(1)}=(\gamma_1,\dots,\gamma_k)$ and
  $\gamma^{(2)}=(\gamma_{k+1},\dots,\gamma_n)$ for each
  $\gamma\in\{\alpha,\beta,\lambda,\mu\}$, because $\lm$ is the disjoint union
  of $\lambda^{(1)}/\mu^{(1)}$ and $\delta^k(\lambda^{(2)}/\mu^{(2)})$. For all
  $k+1\le i\le n$ and $1\le j\le k$, we have
  $h^{\alpha,\beta}_{\lambda,\mu}(i,j)=0$ because
 \[
    \lambda_i-i-\mu_j+j \le \lambda_{k+1}-(k+1)-\mu_k+k <0.
  \]
  By Lemma~\ref{lem:det=detdet}, this implies
  \[
    H_{\lambda,\mu}^{\alpha,\beta} = H_{\lambda^{(1)},\mu^{(1)}}^{\alpha^{(1)},\beta^{(1)}}
    \phi^k\left(H_{\lambda^{(2)},\mu^{(2)}}^{\alpha^{(2)},\beta^{(2)}}\right).
  \]
  Thus, by induction, we may assume $\mu_k<\lambda_{k+1}$ for all $k\in[n-1]$.
  In this case by assumption we have $\alpha,\beta\in\RPar_n$.

  Suppose that $\alpha_k\ge\beta_k$ for some $1\le k\le n$. Then by
  Lemma~\ref{lem:HC=0 mu<la} with $C=\emptyset$ we have
  $H_{\lambda,\mu}^{\alpha,\beta}=0$. Again, by definition,
  $\R_{\lambda,\mu}^{\alpha,\beta}(\emptyset)=0$.

  The remaining case is that $\alpha,\beta\in\RPar_n$, $\mu<\lambda$, and
  $\alpha<\beta$. In this case, by Lemma~\ref{lem:H(1..r)=H} and
  Proposition~\ref{prop:main h},
  \[
    H_{\lambda,\mu}^{\alpha,\beta} = H_{\lambda,\mu}^{\alpha,\beta}([n]) =
    \ovH_{\lambda,\mu}^{\alpha,\beta} = 
    \HH_{\lambda,\mu}^{\alpha,\beta}(\emptyset) = 
    \R_{\lambda,\mu}^{\alpha,\beta}(\emptyset),
  \]
which completes the proof. 
\end{proof}

\section*{Acknowledgments}
The author is grateful to Darij Grinberg for providing his conjecture,
Theorem~\ref{thm:col_flag3}, for fruitful discussions, and for his thorough
reading of the manuscript and providing many useful comments, which
significantly improved the presentation of this paper. The author is
particularly grateful to Darij Grinberg for the idea that improved
Theorem~\ref{thm:col_flag2}. He also thanks Travis Scrimshaw for helpful
discussions.

This work was initiated while the author was participating the 2020 program in
Algebraic and Enumerative Combinatorics at Institut Mittag-Leffler. The author
would like to thank the institute for the hospitality and Sara Billey, Petter
Br\"and\'en, Sylvie Corteel, and Svante Linusson for organizing the program.

This material is based upon work supported by the Swedish Research
Council under grant no. 2016-06596 while the author was in residence at Institut
Mittag-Leffler in Djursholm, Sweden during the winter of 2020.

The author was supported by NRF grants \#2019R1F1A1059081 and \#2016R1A5A1008055.


\begin{thebibliography}{10}

\bibitem{AmanovYeliussizov}
A.~Amanov and D.~Yeliussizov.
\newblock Determinantal formulas for dual {G}rothendieck polynomials.
\newblock \url{https://arxiv.org/abs/2003.03907}.

\bibitem{Chen2002}
W.~Y.~C. Chen, B.~Li, and J.~D. Louck.
\newblock The flagged double {S}chur function.
\newblock {\em J. Algebraic Combin.}, 15(1):7--26, 2002.

\bibitem{GGL2016}
P.~Galashin, D.~Grinberg, and G.~Liu.
\newblock Refined dual stable {G}rothendieck polynomials and generalized
  {B}ender-{K}nuth involutions.
\newblock {\em Electron. J. Combin.}, 23(3):Paper 3.14, 28, 2016.

\bibitem{Gessel_unpublished}
I.~M. Gessel.
\newblock Determinants and plane partitions.
\newblock Unpublished manuscript.

\bibitem{Grinberg_conj}
D.~Grinberg.
\newblock Refined dual stable {G}rothendieck polynomials.
\newblock \url{http://www.cip.ifi.lmu.de/~grinberg/algebra/chicago2015.pdf}.

\bibitem{grinberg14:hopf_algeb_combin}
D.~Grinberg and V.~Reiner.
\newblock {Hopf algebras in combinatorics}.
\newblock {\it Preprint},
  \href{https://arxiv.org/abs/1409.8356v7}{arXiv:1409.8356v7}.

\bibitem{iwao20:free_groth}
S.~Iwao.
\newblock {Free-fermions and adjoint actions on stable $\beta$-Grothendieck
  polynomials}.
\newblock {\it Preprint},
  \href{https://arxiv.org/abs/2004.09499v3}{arXiv:2004.09499v3}.

\bibitem{Kim:JT}
J.~S. Kim.
\newblock Jacobi--{T}rudi formula for refined dual stable {G}rothendieck
  polynomials.
\newblock \url{https://arxiv.org/pdf/2003.00540.pdf}.

\bibitem{LP2007}
T.~Lam and P.~Pylyavskyy.
\newblock Combinatorial {H}opf algebras and {$K$}-homology of {G}rassmannians.
\newblock {\em Int. Math. Res. Not. IMRN}, (24):Art. ID rnm125, 48, 2007.

\bibitem{LS1982p}
A.~Lascoux and M.-P. Sch\"{u}tzenberger.
\newblock Polyn\^{o}mes de {S}chubert.
\newblock {\em C. R. Acad. Sci. Paris S\'{e}r. I Math.}, 294(13):447--450,
  1982.

\bibitem{Loehr_2010}
N.~A. Loehr and J.~B. Remmel.
\newblock {A computational and combinatorial expos\'e of plethystic calculus}.
\newblock {\em Journal of Algebraic Combinatorics}, 33(2):163–198, 2010.

\bibitem{Merzon_2015}
G.~Merzon and E.~Smirnov.
\newblock {Determinantal identities for flagged Schur and Schubert
  polynomials}.
\newblock {\em European Journal of Mathematics}, 2(1):227–245, 2015.

\bibitem{Motegi}
K.~Motegi and T.~Scrimshaw.
\newblock Refined dual Grothendieck polynomials, integrability, and the Schur measure.
\newblock In preparation.

\bibitem{Wachs_1985}
M.~L. Wachs.
\newblock {Flagged Schur functions, Schubert polynomials, and symmetrizing
  operators}.
\newblock {\em Journal of Combinatorial Theory, Series A}, 40(2):276–289,
  1985.

\end{thebibliography}
\end{document}